\documentclass[oneside]{amsart}

\usepackage[letterpaper,body={15.0cm,22.5cm}, mag=1000]{geometry}
\usepackage{amssymb}
\usepackage{amsthm}
\usepackage{amscd}
\usepackage{enumitem}
\usepackage{float}
\usepackage{placeins}
\usepackage{caption}

\usepackage{times}
\usepackage{graphicx}
\usepackage{tikz}
\usepackage{tikz-cd}
\usepackage[all,cmtip]{xy}
\numberwithin{equation}{section}
\theoremstyle{plain}

\newtheorem{cor}[equation]{Corollary}
\newtheorem{lemma}[equation]{Lemma}

\newtheorem{prop}[equation]{Proposition}

\newtheorem{thm}[equation]{Theorem}

\newtheorem*{thma}{Theorem A}
\newtheorem*{thmb}{Theorem B}
\newtheorem*{thmc}{Theorem C}
\newtheorem*{thmd}{Theorem D}
\newtheorem*{thme}{Theorem E}

\newtheorem{rmk}[equation]{Remark}

\theoremstyle{definition}

\newtheorem{remark}[equation]{Remark}

\newcommand{\dlabel}[1]{\ifmmode \text{\ttfamily \upshape [#1] } \else
{\ttfamily \upshape [#1] }\fi \label{#1}}
\newcommand{\A}{\operatorname{A} }
\newcommand{\B}{\operatorname{B} }

\newcommand{\N}{\operatorname{N} }
\newcommand{\Ha}{\operatorname{H} }

\newcommand{\Z}{\operatorname{Z} }

\newcommand{\gen}[1]{\left < #1 \right >}
\newcommand{\Aut}{\operatorname{Aut} }

\newcommand{\Hom}{\operatorname{Hom} }

\newcommand{\Ker}{\operatorname{Ker} }

\newcommand{\Ann}{\operatorname{Ann} }

\newcommand{\CExt}{\operatorname{CExt} }
\newcommand{\Fix}{\operatorname{Fix} }

\begin{document}
\setcounter{page}{1}
\title[Isoclinism of skew left braces]
{Central series' and ($n$)-isoclinism of skew left braces}

\author{Arpan Kanrar}
\address{Harish-Chandra Research Institute, A CI of Homi Bhabha National Institute, Chhatnag Road, Jhunsi, Prayagraj-211 019, India}
\email{arpankanrar@hri.res.in}

\author{Charlotte Roelants}
\address{Department of Mathematics $\&$ Data Science, Vrije Universiteit Brussel, Pleinlaan 2, 1050 Brussel, Belgium.}
\email{Charlotte.Roelants@vub.be}

\author{Manoj K. Yadav}
\address{Harish-Chandra Research Institute, A CI of Homi Bhabha National Institute, Chhatnag Road, Jhunsi, Prayagraj-211 019, India}
\email{myadav@hri.res.in}


\begin{abstract}
The aim of this article is to advance the knowledge on the theory of skew left braces.  We introduce a subclass of skew left braces, which we denote by $\mathcal{I}_n$, $n \ge 1$, such that elements of the  annihilator and lower central series' interact `nicely' with respect to commutation.  That allows us to define a concept of   $n$-isoclinism of skew left braces  in $\mathcal{I}_n$, by using a concept of brace commutator words, which we have introduced.  We prove  results on $1$-isoclinism (isoclinism) of skew left braces analogous to important results in group theory.  For any two symmetric $n$-isoclinic skew left braces $A$ and $B$, we prove that, there exist skew left braces $C$ and $R$ such that both $A$ and $B$ are $n$-isoclinic to both $C$ and $R$ and (i) $A$ and $B$ are quotient skew left braces of $C$; (ii) $A$ and $B$ are sub-skew left  braces of $R$. Connections between a skew left brace and the group which occurs as a natural semi-direct  product of additive and multiplicative groups of the skew  left brace are investigated, and it is proved that $n$-isoclinism is preserved from braces to groups. We also show that various nilpotency concepts on skew left braces  are invariant under $n$-isoclinism. 
 \end{abstract}
\maketitle

\section{Introduction}

The concept of isoclinism ($1$-isoclinism) of groups was introduced by P. Hall in \cite{Hall40} in quest of classification of finite $p$-groups. In the same paper, Hall implicitly introduced the concept of $n$-isoclinism. Recently the concept of isoclinism ($1$-isoclinism) of  skew left braces was introduced by  Letourmy and Vendramin \cite{LV23} as a part of the important project of classification of skew left braces. This article is aimed  to make further prominent advancement on $n$-isoclinism ($n \ge 1$) of skew left braces to strengthen the classification project. We introduce the concept of a commutator word map in skew left braces and a skeleton of  lower central series of a skew left brace, which, we expect, will open new avenues of research in this area.

A non-empty set $A$ with two group structures, $(A, \cdot)$ and $(A, \circ)$,  is said to be a \emph{skew left brace} if
 \begin{equation}
 a \circ (b \cdot c) =  (a \circ b) \cdot a^{-1} \cdot  (a \circ c)
 \end{equation}
holds true for all $a, b, c \in A$, where $ a^{-1}$ denotes the  inverse of $a$ in $(A, \cdot)$. Whenever we want to emphasize on the binary operations of the skew left brace $A$, we shall  denote it by  a triple $(A, \cdot, \circ)$.  Otherwise we shall simply denote it by $A$. We call  $(A, \cdot)$ the \emph{additive group} and $(A, \circ)$ the \emph{multiplicative  group} of the skew left brace $(A, \cdot, \circ)$. A skew left brace $(A, \cdot, \circ)$ is said to be a \emph{left brace} if $(A, \cdot)$ is an abelian group.  A  skew right brace is defined analogously. In this article we mainly consider  skew left braces. So, we will mostly suppress the word `left' and only say skew brace(s).

 The concept of a left brace was introduced by Rump \cite{Rump07} in 2007 in connection with non-degenerate involutive set-theoretic solutions of the  Yang-Baxter equation.   Guarnieri and Vendramin \cite{GV17} generalized this in 2017 to skew left braces to encompass non-involutive set-theoretic solutions too. The subject received tremendous attention of the mathematical community after the work of Rump, as evident from the list of papers in the bibliography and the references therein. Interest in the study of set-theoretic solutions of the  Yang-Baxter equation was intrigued by the paper \cite{Drinfeld92}  of Drinfeld, published in 1992, and its connection with algebraic structures in the paper \cite{ESS99} published in 1999. It is already very well established by now that any advancement on the knowledge on skew braces will prominently contribute to the theory of the  Yang-Baxter equation. With this motivation in the background, we now introduce the work done in the paper.
  
 Schur covers of finite skew braces were introduced and studied in \cite{LV24}. The authors proved that any two Schur covers of a finite skew brace are isoclinic.  We generalize this result as follows. Let $K$ be a skew brace acting trivially on an abelian group  $A$ (viewed as a trivial brace) and $\Ha^2_b(K, A)$ be the second brace cohomology group of $K$ with coefficients in $A$ (see Section 3 for details). Let $[(\alpha, \mu)] \in \Ha^2_b(K, A)$ be a fixed cohomology class and $G = A \times_{(\alpha, \mu)} K$ be the corresponding annihilator extension skew brace of $K$ by $A$. Set $\hat{A} := \Hom(A, \mathbb{C}^{\times})$. Then there exists a group homomorphism $\varepsilon_G : \hat{A} \to \Ha^2_b(K, \mathbb{C}^{\times})$ given by $\varepsilon_G(\lambda) = [(\lambda \alpha, \lambda \mu)]$ for all $\lambda \in \hat{A}$. Our first  main result, which follows from  Theorem \ref{thm:gener:Schur:cov} (which is more general, but difficult to state here)  and Theorem \ref{thm:Tappe},  is 
 
 \begin{thma}
 Let a finite skew brace $K$ act trivially on abelian groups $A$ and $B$.  Let $[(\alpha, \mu)] \in \Ha^2_b(K, A)$ and $[(\beta, \nu)] \in  \Ha^2_b(K, B)$ with corresponding annihilator extension skew braces $G = A \times_{(\alpha, \mu)} K$ and $H = B \times_{(\beta, \nu)} K$ respectively.  Then $G$ and $H$ are isoclinic  if $\varepsilon_G(\hat{A}) = \varepsilon_H(\hat{B})$.
 \end{thma}

  In Section 4, we define the concept of brace commutator words for skew braces analogous to the commutator words in group theory. We construct a filtration of strong left ideals: $A = A_{(0)}  \ge A_{(1)} \ge A_{(2)} \ge \cdots$ of a  given skew brace $A$, which lies  below the lower central series of ideals: $A = \Gamma_1(A) \ge \Gamma_2(A) \ge \Gamma_3(A) \ge  \cdots$ of $A$, that is, $A_{(i)} \le \Gamma_{i+1}(A)$ for all $i \ge 0$,  and prove that the length of the lower central series is determined by the length of this  filtration.  For a positive integer $n$, a concept of $n$-isoclinism is defined in the subclass $\mathcal{I}_n$ of the class of all skew braces $\mathcal{SB}$ using the brace commutator words, and then obtain equivalent conditions that allow a skew brace to fall in  $\mathcal{I}_n$. More precisely,  $\mathcal{I}_n$ is the largest subclass of skew braces in which we can talk of $n$-isoclinism. We remark that $\mathcal{I}_1 = \mathcal{SB}$  (proved in  \cite{LV23}).
  
  A skew brace $(A, \cdot, \circ)$ is said  to be \emph{symmetric} if $(A, \circ, \cdot)$ is also a skew brace.  In Section 5, we prove that  all symmetric skew braces fall in  $\mathcal{I}_n$ for all $n \ge 1$. The following result, analogous to \cite[Theorem 1.4]{Bioch78}, is also proved in this section.
  
 \begin{thmb}
  Let $A, B$ be symmetric  skew braces. Then $A$ and $B$ are $n$-isoclinic if and only if there exist  skew braces $C$, $N_1$, $N_2$, $C_A$ and $C_B$ such that $A \cong C/N_2$ and $B \cong C/N_1$ and the following two equivalent conditions hold:
\begin{enumerate}
\item  $C$ is $n$-isoclinic to both $A$ and $B$.
\item $C_B \cong C \cong C_A$, and $C$ is $n$-isoclinic to both $C/N_2 \times C/\Gamma_{n+1}(C)$ and  $C/N_1 \times C/\Gamma_{n+1}(C)$, where $C_B$ and $C_A$ are sub-skew  braces of $C/N_2 \times C/\Gamma_{n+1}(C)$ and  $C/N_1 \times C/\Gamma_{n+1}(C)$ respectively.
\end{enumerate}
 \end{thmb}
 
 We continue exploring symmetric  skew braces in Section 6. We first prove  the following embedding theorem for $1$-isoclinic skew braces in the full class $\mathcal{SB}$ of skew braces.
 
 \begin{thmc}
 Let $A$ and $B$ be isoclinic skew braces. Then there exists a skew  brace $W$ with sub-skew braces $\bar{A}$ and $\bar{B}$ such that $A\cong \bar{A},~B\cong \bar{B},$ and $\bar{A}  \Ann(W)=\bar{A}\circ \Ann(W) = W = \bar{B} \circ \Ann(W) = \bar{B}  \Ann(W)$.
 \end{thmc}
 
 As a consequence we obtain  \cite[Proposition 3.10]{LV23} which states that isoclinism between two skew braces $(A, \cdot, \circ)$ and $(B, \cdot, \circ)$, ensures isoclinism  between $(A, \cdot)$ and $(B, \cdot)$, and  $(A, \circ)$ and $(B, \circ)$. We then obtain the following embedding theorem for $n$-isoclinism of symmetric skew braces.
 
\begin{thmd}
  Let $A, B$ be two $n$-isoclinic ($n \ge 1$) symmetric skew braces. Then there exists a skew brace $R := \tilde{A} \times N$ with the following properties:
\begin{enumerate}
\item $N$ is a nilpotent skew brace of nilpotency class $n$;
\item $\tilde{A}$ is brace isomorphic to $A$, and $R=\bar{A}\Ann_n(R)$;
\item $R$ admits a sub-skew brace $\bar{B}$  isomorphic to $B$ such that  $R=\bar{B}\Ann_n(R)$,
\end{enumerate}
where $\bar{A}:=\tilde{A}\times \{1\}$. In particular, both $A$ and $B$ are $n$-isoclinic to $R$.
\end{thmd}

As a consequence we obtain the result:   $n$-isoclinism between two symmetric skew braces $(A, \cdot, \circ)$ and $(B, \cdot, \circ)$, ensures $n$-isoclinism  between $(A, \cdot)$ and $(B, \cdot)$, and  $(A, \circ)$ and $(B, \circ)$. 

\begin{remark}
The results of Section 5 and Section 6 also hold true for $\lambda$-homomorphic skew braces (see Section 2 for the definition). Since the proofs  mostly follow verbatim  the ones on symmetric skew braces, we decided to work out the details only for the latter class.
\end{remark}
  
In the final section, Section 7, we investigate properties of skew braces which are invariant under $n$-isoclinism.  Let $(G, \cdot, \circ)$ be a skew brace and $\lambda^{op} : (G, \circ) \to \Aut(G, \cdot)$ be the group action of $(G, \circ)$ on $(G, \cdot)$ given by the $\lambda$ map of $(G, \cdot, \circ)$. Set $\Lambda_G := (G, \cdot) \rtimes_{\lambda^{op}}  (G, \circ)$, the semidirect product of $(G. \cdot)$ by $(G, \circ)$ under the action $\lambda^{op}$. Our next result is
  
  \begin{thme}
 Let $A$ and $B$ be two  $n$-isoclinic  skew braces in $\mathcal{I}_n$ such that $A_{(n)} = \Gamma_{n+1}(A)$ and $B_{(n)} = \Gamma_{n+1}(B)$. Then the groups $\Lambda_A$ and $\Lambda_B$ are $n$-isoclinic. 
  \end{thme}
  
  We then show that various nilpotency  concepts on skew braces are invariant under $n$-isoclinism. 
  
  We conclude this section by setting some notations. Note that additive and multiplicative identities of any skew brace coincide, which we denote by $1$. A skew brace having cardinality $1$ is also denoted by $1$, instead of $\{1\}$.  Inverse of  an element $a \in A$, when considered in $(A, \cdot)$  (respectively in $(A, \circ)$)  is denoted by $a^{-1}$  (respectively by $\bar{a}$). The group commutator of two elements $a$ and $b$ of a skew brace $A$ can be defined in two ways, one in $(A, \cdot)$ and the other in $(A, \circ)$. In $(A, \cdot)$, we denote it by $[a, b]^{\cdot} := a \cdot b \cdot a^{-1} \cdot b^{-1}$, and in $(A, \circ)$, we denote it by $[a, b]^{\circ} := a \circ b \circ \bar{a} \circ \bar{b}$. The commutator  $[a, b]^{\cdot}$ is also denoted by $\gamma_{\cdot}(a, b)$. The commutators of larger weights are defined iteratively. By $a^b$, we mean the conjugate $b \cdot a \cdot b^{-1}$ of $a$ by $b$ in $(A, \cdot)$.  We write $B \le A$ to denote that $B$ is a sub-skew brace of the skew  brace $A$.  To avoid confusion with the notation $\bar{a}$ as an inverse in $(A, \circ)$, we will denote a coset representative  of the coset $Ia$ by  $\tilde{a}$ or $\hat{a}$, where  $I$ is an ideal of a skew brace $A$ and $a \in A$. However, the quotient skew brace  $A/I$ of $I$ in $A$ may be denoted by $\bar{A}$. The class of all skew left braces will be denoted by $\mathcal{SB}$.
      
%%%%%%%%%%%%%%%%%%%%%%%%%%%%%%%%%%%%%%%%%%%%%%%%%%%%%%%%%%%%%%%%

\section{Preliminaries}

A  triple $(A, \cdot, \circ)$, where $(A, \cdot)$ and $(A, \circ)$ are  groups,    is said to be a \emph{skew left brace} if
 \begin{equation}
 a \circ (b \cdot c) =  (a \circ b) \cdot a^{-1} \cdot  (a \circ c)
 \end{equation}
 for all $a, b, c \in A$, where $ a^{-1}$ denotes the  inverse of $a$ in $(A, \cdot)$.  For a skew left brace $(A, \cdot, \circ)$, it was  proved in \cite{GV17} that the map
$$
\lambda  :  (A, \circ) \to \Aut \, (A, \cdot),~~a \mapsto \lambda_a
$$
is a group homomorphism, where $\Aut \, (A, \cdot)$ denotes the automorphism  group of $(A, \cdot)$ and $\lambda_a$ is given by $\lambda_a(b) = a^{-1} \cdot (a \circ b)$ for all $a, b \in A$.  Define 
$$a*b = \lambda_a(b) \cdot b^{-1}$$ for all $a, b \in A$.

For two elements $a, b \in A$, we define $[a, b]^{\cdot} = a \cdot b \cdot a^{-1} \cdot b^{-1}$ and  $[a, b]^{\circ} = a \circ b \circ \bar{a} \circ \bar{b}$, where for any $x \in A$, $\bar{x}$ denotes the inverse of $x$ in $(A, \circ)$. From now onwards we mostly suppress the use of `$\cdot$' as an operation in the computations. It is straightforward to see that

$$[ab,c]^{.} = [a,[b,c]^{\cdot}]^{\cdot} [b,c]^{\cdot} [a,c]^{\cdot}$$
and 
$$[a, bc]^{\cdot} = [a,b]^{\cdot} [b,[a,c]^{\cdot}]^{\cdot} [a,c]^{\cdot}$$
for all $a, b, c \in A$.
The same happens under the operation `$\circ$'.

A subset $I$ of a skew left brace $(A, \cdot, \circ)$ is said to be a \emph{ left ideal } of $(A, \cdot, \circ)$ if $I$ is a  subgroup of $(A, \cdot)$ and $\lambda_a(c) \in I$ for all $a \in A$ and $c \in I$. It turns out that a left ideal $I$ of $(A, \cdot, \circ)$ is a subgroup of $(A, \circ)$. A left ideal $I$ of $A$ is said to be  a \emph{a strong left ideal} if $I$ is normal in $(A, \cdot)$. A left ideal $I$ is said to be an \emph{ideal} of  $(A, \cdot, \circ)$ if $I$ is normal in both $(A, \cdot)$ and  $(A, \circ)$. 

The following result is well known (see, e.g., \cite[Proposition 7]{BEJP24}).
\begin{lemma}\label{lemma1}
Let $I$ be an ideal and $B$ a sub-skew left brace of a skew  brace  $A := (A, \cdot, \circ)$. Then $BI = IB$ is a  sub-skew left brace of $A$.
\end{lemma} 
%\begin{proof}
%That $BI = IB$  and $BI$ is closed under `$\cdot$' is clear from the normality of $I$ in $(A, \cdot)$.  Let $b \in B$ and $c \in I$.  Notice that $c^{-1}b^{-1} = b^{-1}(bc^{-1}b^{-1}) \in the inverse of $bc$  under the operation `$\cdot$'. Thus $(BI, \cdot)$ is subgroup of $(A, \cdot)$. Now we prove that $BI$ is a subgroup of $(A, \circ)$. First we see that 
%\begin{eqnarray*}
%(bc)\circ(b_1c_1) &=& b \circ \lambda^{-1}_b(c) \circ b_1 \circ \lambda^{-1}_{b_1}(c_1)\\
%&=& b \circ b_1 \circ (\bar{b}_1 \circ \lambda^{-1}_b(c) \circ b_1)  \circ  \lambda^{-1}_{b_1}(c_1)\\
%&=& (b \circ b_1) \lambda_{b \circ b_1}\big(\bar{b}_1 \circ \lambda^{-1}_b(c) \circ b_1  \circ  \lambda^{-1}_{b_1}(c_1)\big)
%\end{eqnarray*}
%lies in $BI$. It is easy to see that $\bar{b}\lambda^{-1}_b(b \circ \overline{\lambda^{-1}_b(c)} \circ \bar{b})$ is the inverse of $bc$ under the operation `$\circ$'. This completes the proof.
%\end{proof}

\begin{cor}
Let $I$ be an ideal and $J$ a  left ideal  of a skew brace  $A$. Then $JI = IJ$ is a   left ideal  of $A$. Moreover, if $J$ is also an ideal of $A$, then $JI = IJ$ is an ideal  of $A$.
\end{cor}

Let $A := (A, \cdot, \circ)$ be a skew  brace. Set $\Ann_0(A) = 1$, 
\begin{eqnarray*}
\Ann_1(A) = \Ann(A) &:=& \Ker(\lambda) \cap \Z(A, \cdot) \cap \Z(A, \circ)\\
&=& \Fix(\lambda) \cap \Z(A, \cdot) \cap \Z(A, \circ)\\
&=&  \{a \in A \mid a*b = 1, [a, b]^{\cdot} = 1, [a, b]^{\circ} = 1, \mbox{ for all } b \in A\}\\
&=&  \{a \in A \mid a*b = 1, b*a = 1,  [a, b]^{\cdot} = 1, \mbox{ for all } b \in A\}
\end{eqnarray*}
and for $n \ge 2$, define $\Ann_n(A)$, the $n$th annihilator  of $A$, by
$$\Ann_n(A) / \Ann_{n-1}(A) = \Ann\big(A /  \Ann_{n-1}(A)\big).$$
This is possible because $\Ann_{n-1}(A)$ is an ideal of $A$ \cite[Definition 2.2]{JAV23}.  Notice that 
$$\Ann_n(A) = \{a \in A \mid a*b, b*a,  [a, b]^{\cdot}  \in \Ann_{n-1}(A) \mbox{ for all } b \in A\}.$$
This is an ascending series of ideals of $A$, which is called the  \textit{annihilator series}. We shall view this series  analogous to the upper central series of a group.

Let $A := (A, \cdot, \circ)$ be a skew  brace and $\gamma_n(A, \cdot)$ and  $\gamma_n(A, \circ)$, respectively, denote the $n$th terms of the lower central series of $(A, \cdot)$ and $(A, \circ)$.  For the left ideals $I$ and $J$ of a skew brace $A$, define 
 $$I*J := \langle x*y \mid x \in I, y \in J\rangle^{\cdot},$$ 
 the subgroup generated by the set $\{x*y \mid x \in I, y \in J\}$ in $(A, \cdot)$.
 Set $A^1 = A$ and, recursively, define 
$$A^n :=  A*A^{n-1} = \gen{a*b \mid a \in A, b\in A^{n-1}}^{\cdot}.$$
 It is well known that $A^2$ is an ideal of $A$. But $A^n$, $n \ge 3$, in general is not an ideal of $A$. However, $A^n$ is a left ideal of $A$.  Further, set $A^{(1)} = A$ and, recursively, define 
$$A^{(n)} :=  A^{(n-1)}*A = \gen{b*a \mid a \in A, b\in A^{(n-1)}}^{\cdot}.$$
 It is proved in \cite[Proposition 2.1]{CSV19} that $A^{(n)}$ is an ideal of $A$.
 
 We consider the lower central series of a skew brace defined in \cite{BJ23}. 
  Set $\Gamma_1(A)  := A$ and define
$$\Gamma_n(A) := \gen{A * \Gamma_{n-1}(A), \Gamma_{n-1}(A) * A, [A, \Gamma_{n-1}(A)]^{\cdot}}^{\cdot},$$
where $[A, \Gamma_{n-1}(A)]^{\cdot}$ denotes the subgroup of $(A, \cdot)$ generated by the set $\{[a, u]^{\cdot} \mid a \in A, u \in \Gamma_{n-1}(A)\}$.
The authors of  \cite{BJ23} start with $0$ index, but we start with $1$. $\Gamma_n(A)$ is an ideal of $A$ and  $\Gamma_{n+1}(A) \le \Gamma_{n}(A)$ for all integers  $n \ge 1$. We shall view this series of ideals of a skew left  brace analogous to the lower central series of a group.

A skew brace $A$ is said to be  {\em nilpotent} if there exists an integer $n$ such that  $\Ann_n(A) = A$. The least such integer is called the  {\em nilpotency class} of $A$.  The authors of \cite{BJ23} call such skew braces centrally nilpotent and prove that   $\Ann_n(A) = A$ if and only if $\Gamma_{n+1} = 1$ (\cite[Theorem 2.8]{BJ23}). 

The following interesting observation in comes very handy at several places.
\begin{lemma}\label{lemma2}
For any skew brace $A$, $\gamma_n(A, \circ)$ is contained in  $\Gamma_n(A)$. 
\end{lemma}
\begin{proof}
We first prove the result for $n=2$ and then use induction. It suffices to work on the generators.  We have 
\begin{eqnarray*}
x \circ y \circ \bar{x} \circ \bar{y} & = & x \circ y \circ ( \bar{x}  \lambda_{\bar{x}}( \bar{y}))\\
&=& x \circ \big(y \lambda_y( \bar{x}  \lambda_{\bar{x}}( \bar{y}))\big)\\
&=& x \lambda_x (y \lambda_y( \bar{x}  \lambda_{\bar{x}}( \bar{y})))\\
&=& \lambda_x (\bar{x}^{-1}) \lambda_x \big(\lambda_y(\bar{y}^{-1}) \lambda_y( \bar{x}  \lambda_{\bar{x}}( \bar{y}))\big)\\
&=& \lambda_x\big(\bar{x}^{-1} \lambda_y(\bar{y}^{-1}  \bar{x}  \lambda_{\bar{x}}( \bar{y}))\big)\\
&=& \lambda_x\big(\bar{x}^{-1} \lambda_y([\bar{y}^{-1}, \bar{x}]^{\cdot}) \lambda_y( \bar{x}) \lambda_y (\bar{y}^{-1} \lambda_{\bar{x}}( \bar{y}))\big)\\
&=& \lambda_x\big(\bar{x}^{-1} (\lambda_y([\bar{y}^{-1}, \bar{x}]^{\cdot}) \lambda_y( \bar{x})\bar{x}^{-1}) \bar{x} \lambda_y (\bar{y}^{-1} \lambda_{\bar{x}}( \bar{y}))\big),
\end{eqnarray*}
which lies in $\Gamma_2(A)$.   By induction assume that  $\gamma_{n-1}(A, \circ) \subseteq \Gamma_{n-1}(A)$. Let $u = [v, a]^{\circ} \in \gamma_{n}(A, \circ)$, where $v \in \gamma_{n-1}(A, \circ) \le \Gamma_{n-1}(A)$ and $a \in A$. Now, the above computation, in a little altered form, gives
\begin{equation*}
u = [v, a]^{\circ} = \lambda_v\big(\bar{v}^{-1}\bar{a}^{-1}  \bar{v}  \bar{a} (c^{-1} \lambda_a(c)) \lambda_a (\bar{a}^{-1} \lambda_{\bar{v}}( \bar{a}))\big),
\end{equation*}
where $c:=\bar{a}^{-1}  \bar{v}  \bar{a}$.  Notice that  $\bar{v}^{-1}\bar{a}^{-1}  \bar{v}  \bar{a} = ([\bar{a}^{-1}, \bar{v}^{-1}]^{\cdot})^{-1}  \in [A, \Gamma_{n-1}(A)]^{\cdot}$,  $c^{-1} \lambda_a(c) = (\lambda_a(c^{-1})c) ^{-1} \in A*\Gamma_{n-1}(A)$ and  $\bar{a}^{-1} \lambda_{\bar{v}}( \bar{a}) = ( \lambda_{\bar{v}}( \bar{a}^{-1}) \bar{a})^{-1} \in \Gamma_{n-1}(A)*A$ as $c, v \in \Gamma_{n-1}(A)$.  Hence $\gamma_{n}(A, \circ) \subseteq \Gamma_{n}(A)$, and the proof is complete.
\end{proof}

\begin{rmk}\label{rmk1}
In the second part of the preceding proof we indeed proved that $\gamma_n(A, \circ) \subseteq [\Gamma_{n-1}(A), A]^{\circ} \subseteq \Gamma_{n}(A)$.
\end{rmk}

We now introduce two subclasses of skew braces, called symmetric and $\lambda$-homomorphic skew braces. As mentioned in the introduction, a skew brace $(A, \cdot, \circ)$ is said to be symmetric if $(A, \circ, \cdot)$ is also a skew brace. It turns out that a skew brace $(A, \cdot, \circ)$ is symmetric if and only if  $\lambda_{ab} = \lambda_{b \circ a}$ for all $a, b \in A$ (\cite[Proposition 5.2]{BNY22}). Citing the same reference  \cite{BNY22}, we say that a skew braces $A$ is $\lambda$-homomorphic if $\lambda_{a \cdot b} = \lambda_{a \circ b}$ for all $a, b \in A$. It is not difficult to see that  $(A, \cdot, \circ)$ is symmetric if and only if the image of $(A, \circ)$ under the group homomorphism $\lambda$ (defined above)  is an abelian subgroup of $\Aut(A, \cdot)$. A GAP \cite{GAP} computation, using the package  Yang-Baxter \cite{VK22}, shows that majority of $\lambda$-homomorphic skew braces are symmetric. Out of total 400 skew braces of oder 36, 124 are $\lambda$-homomorphic and 156 are symmetric. But there are only 115 skew braces of order 36 which are both symmetric and $\lambda$-homomorphic. Similarly, there are 305 $\lambda$-homomorphic skew braces of order 72 which are not symmetric. Total number of  $\lambda$-homomorphic skew braces of order 72 is 4124. Therefore, although seems small in number, there are $\lambda$-homomorphic skew braces which are not symmetric. Further relation between these two nice classes of skew braces is a matter of further investigation.

We conclude this section by defining the concept of isoclinism of skew  braces introduced in \cite{LV23}. For a skew left brace $B$, it is easy to check that the following maps are well defined:
\begin{align*}
\phi^B_{\cdot}&:(B/{\Ann B})^2\to \Gamma_2(B), \hspace{1.5cm}(\tilde{a},\tilde{b})\mapsto [a,b]^{\cdot},\\
\phi^B_\ast&:(B/{\Ann B})^2\to \Gamma_2(B),\hspace{1.5cm}(\tilde{a},\tilde{b})\mapsto a\ast b,
\end{align*}
where $\tilde{a}, \tilde{b} \in B/\Ann(B)$. 
We say that two skew  braces $A$ and $B$ are {\em isoclinic} if there exist  brace isomorphisms $\xi : A/{\Ann A} \to B/{\Ann B}$ and $\theta : \Gamma_2(A) \to \Gamma_2(B)$ such that the diagram
\begin{equation}\label{dia:isoclinism}
\begin{tikzcd}
\Gamma_2(A) \arrow[d,"\theta"] &(A/{\Ann A})^2\arrow[l,"\phi_{\cdot}^A"']\arrow[r,"\phi_\ast^A"]\arrow[d,"\xi \times \xi"] &\Gamma_2(A) \arrow[d,"\theta"]\\
\Gamma_2(B) &(B/{\Ann B})^2\arrow[l,"\phi_{\cdot}^B"]\arrow[r,"\phi_\ast^B"'] & \Gamma_2(B)
\end{tikzcd}
\end{equation}
commutes. The pair $(\xi,\theta)$ is called a  {\em skew  brace  isoclinism}. This concept will be generalized to the concept of $n$-isoclinism below, $n \ge 1$.

%%%%%%%%%%%%%%%%%%%%%%%%%%%%%%%%%%%%%%%%%%%%%%%%%%%%%%%

\section{Isoclinism of annihilator extensions of skew braces}

Let $K:= (K, \cdot, \circ)$ be a skew  brace and $A$ an abelian group.  The concept of action of $K$ on $A$ (viewed as a trivial brace) and second brace cohomology group of $K$ with coefficients in $A$ was studied in \cite{NY24}. We here consider the trivial action of $K$ on $A$.
A {\em factor set} or {\em$2$-cocycle} of  $K$ with coefficients in $A$  is a pair  of maps $(\alpha,\mu)$,  $\alpha,\mu : K \times K \to A$, satisfying the following conditions for all $x, y, z \in K$:
\begin{align}
&\alpha(1,y)=\alpha(x,1)=1,\label{eq:bCocycle1}\\
&\mu(1,y) =\mu(x,1)=1,\label{eq:bCocycle2}\\
&\alpha(y,z) \alpha(x \cdot y, z)^{-1}\alpha(x,y \cdot z)\alpha(x,y)^{-1}=1,\label{eq:bCocycle3}\\
&\mu(y,z)\mu(x \circ y,z)^{-1}\mu(x,y \circ z)\mu(x,y)^{-1}=1, \label{eq:bCocycle4}\\
&\alpha(y,z) \alpha(x, x^{-1})  \alpha(x \circ y, x^{-1})^{-1} \alpha((x \circ y) \cdot x^{-1}, x \circ z)^{-1} = \mu(x,z)\mu(x,y \cdot z)^{-1} \mu(x,y).\label{eq:compati}
\end{align}

As observed in \cite{LV24}, using \eqref{eq:bCocycle3} for  the elements $x \circ y, x^{-1}, x \circ z$  and then for the elements $x , x^{-1}, x \circ z$, \eqref{eq:compati} can equivalently be written more neatly as 
\begin{align}
\alpha(y,z)\alpha(x\circ y,\lambda_x(z))^{-1} \alpha(x,\lambda_x(z)) = \mu(x,y)\mu(x,y \cdot z)^{-1}\mu(x,z).\label{eq:equivCompati}
\end{align}

A $2$-cocycle $(\alpha, \mu)$ of  $K$ with coefficients in $A$ is said to be a {\em $2$-coboundary} if there exists  a map $h:K \to A,$ such that $h(1) = 1$ and the following hold for all $x, y \in K$:
\begin{align*}
\alpha(x, y) &=  h(x)h(y)h(x \cdot y)^{-1},\\
\mu(x, y) &=h(x)h(y)h(x\circ y)^{-1}.
\end{align*}
Let us denote the set of all $2$-cocycles by $\Z^2_b(K, A)$ and the set of all $2$-coboundaries by  $\B^2_b(K, A)$. Notice that $\Z^2_b(K, A)$ is an abelian group, under component-wise operation, and $\B^2_b(K, A)$ is its subgroup. The quotient group 
$\Ha^2_b(K, A) := \Z_b^2(K, A)/ \B^2_b(K, A)$ is called the {\em second cohomology group} of the skew left brace $K$ with coefficients in  the abelian group  $A$ having trivial action.

Let $K$ and $A$ be  skew  braces. A skew  brace $G$ is said to be a {\em brace extension} of $K$ by $A$ if the following sequence of skew  braces is exact
\begin{equation}\label{eq:annExt}
\begin{tikzcd}
\mathcal{E}_G : 1 \arrow[r] &A \arrow[r,"i"] &G \arrow[r,"\pi"] &K \arrow[r] & 1.
\end{tikzcd}
\end{equation}
For simplicity, we shall denote such an extension by $\mathcal{E}_G$ itself. An extension $\mathcal{E}_G$ of $K$ by $A$ is said to be an \textit{annihilator extension} if $i(A) \le \Ann(G)$. Let $\CExt(K, A)$ denote the set of all annihilator extensions of $K$ by $A$. Then it follows form  \cite[Corollary 3.7]{NY24} that $\Ha^2_b(K, A)  \leftrightarrow \CExt(K, A)$. We elaborate on this a little more here. For a given cohomology class $[(\alpha, \mu)] \in \Ha^2_b(K, A)$, we can define a skew  brace $G := (A \times_{(\alpha, \mu)} K, \cdot, \circ)$, where
\begin{eqnarray*}
(a_1, k_1) \cdot (a_2, k_2) &=& (a_1a_2 \alpha(k_1, k_2), k_1 \cdot k_2);\\
(a_1, k_1) \circ (a_2, k_2) &=& (a_1a_2 \mu(k_1, k_2), k_1 \circ k_2)
\end{eqnarray*}
for all $(a_i, k_i) \in A \times K$, $i = 1, 2$.  It turns out that $G$ is an annihilator extension of $K$ by $A$ with $i(a) = (a, 1)$ and $\pi(a, k) = k$. On the other hand, if we start with an annihilator extension $\mathcal{E}_G$, we can define a $2$-cocycle as follows. Let $t : K \to A$ be transversal of $\pi$ such that $t(1) = 1$. Then define  $\alpha, \mu : K \times K \to A$ by setting
\begin{align*}
\alpha(k_1, k_2) &= t(k_1)\cdot  t(k_2)\cdot t(k_1 \cdot k_2)^{-1};\\
\mu(k_1, k_2) &= t(k_1) \circ t(k_2) \circ \overline{t(k_1 \cdot k_2)}
\end{align*}
for all $k_1, k_2 \in K$. For proofs and other details, please refer to \cite[Section 3]{NY24}.

Let $K$ be a skew brace and $A$  an abelian group viewed as a trivial brace. Let $\Ha_b^2(K,A)$ be the second cohomology group of $K$ with coefficients in $A$ under trivial action. Let $G := A \times_{(\alpha, \mu)} K$ be the brace extension of $K$ by $A$ determined by the cohomology class  $[(\alpha, \mu)] \in \Ha_b^2(K,A)$. So we get the extension 
\[1 \to A \to G \to K \to 1\]
such that  $A$ embeds in $\Ann(G)$. For simplicity of notation, we view $A$ as an ideal of $G$.  Let $C$ be an abelian group, again viewed as a trivial brace.  For any homomorphism $\lambda : A \to C$, we can define maps $\alpha', \mu'  : K \times K \to C$ given by
\begin{eqnarray*}
\alpha'(k_1, k_2) &:=&  \lambda(\alpha(k_1, k_2));\\
\mu'(k_1, k_2) &:=& \lambda(\mu(k_1, k_2))
\end{eqnarray*}
for all $k_1, k_2 \in K$. A routine check reveals that $(\alpha', \mu')$ is a $2$-cocycle of $K$ with coefficients in $C$, viewing $C$ as a trivial $K$-module, which, in turn, defines a group homomorphism $\varepsilon_{G} : \Hom(A, C) \to \Ha_b^2(K,C)$ given by
\[\varepsilon_G(\lambda) := [(\alpha', \mu')]\]
for all $\lambda \in \Hom(A, C)$. For a given annihilator extension 
\[1 \to A \to G \to K \to 1\]
of skew left braces, the homomorphism $\varepsilon_G$, defined using the corresponding cohomology class,  is termed as the {\em transgression} homomorphism in the literature (see \cite[pages 621-622]{LV24}).  We consider a specific situation when $K$ is finite and $C = \mathbb{C}^{\times}$, the multiplicative group of all non-zero complex numbers. More precisely, we deal with the homomorphism  $\varepsilon_{G} : \hat{A} \to \Ha_b^2(K,\mathbb{C}^{\times})$, where for any abelian group $D$, $\hat{D} := \Hom(D, \mathbb{C}^{\times})$.

\begin{lemma}\label{lem:BraceEpsilonKernel}
Let $A$ be an ideal of a skew  brace $G$ such that $A \le \Ann(G)$. Then $\Ker(\varepsilon_G) \cong (A \cap \Gamma_2(G))^{\perp}$ and $\varepsilon_G(\hat{A})  \cong \widehat{A \cap \Gamma_2(G)}$.
\end{lemma}

\begin{proof} 
Let $[(\alpha, \mu)] \in \Ha_b^2(K,A)$ be the cohomology class corresponding to the extension
\[1 \to A \to G \to K \to 1.\]
So $G \cong A \times_{(\alpha, \mu)} K$.
Let $\lambda \in \hat{A}$ and $[(\alpha', \mu')]  \in \Ha^2_b(K, \mathbb{C}^{\times})$ be the cohomology class induced by $\lambda$  as explained above.

Let $\lambda\in \Ker(\varepsilon_G)$, that is,  $(\alpha', \mu') \in \B_b^2(K, \mathbb{C}^{\times})$. Then there is a map $h: K \to \mathbb{C}^\times$ with $h(1)=1$ and 
\begin{align*}
\alpha'(k_1, k_2) = h(k_1) h(k_2)  h(k_1 \cdot k_2)^{-1}\\
\mu'(k_1, k_2) = h(k_1) h(k_2)  h(k_1 \circ k_2)^{-1}
\end{align*}
for all $k_1, k_2 \in K$. We are now going to extend $\lambda$ to a brace homomorphism $\tau : G \to \mathbb{C}^\times$ as follows. For any $(a, k) \in G$, define $\tau$ by
\begin{align*}
\tau(a, k)=\lambda(a)h(k).
\end{align*}
Clearly $\tau$ is well-defined.  Now the equations
\begin{align*}
\tau((a_1, k_1) \cdot (a_2, k_2))&=\tau(a_1 a_2 \alpha(k_1, k_2), k_1 \cdot k_2)\\
&= \lambda(a_1 a_2 \alpha(k_1, k_2)) h(k_1 \cdot k_2)\\
\tau(a_1, k_1)\tau(a_2, k_2)&=\lambda(a_1)h(k_1)\lambda(a_2)h(k_2)\\
&=\lambda(a_1 a_2)\alpha'(k_1, k_2) h(k_1 \cdot k_2)\\
&=\lambda(a_1 a_2 \alpha(k_1, k_2)) h(k_1 \cdot k_2),
\end{align*}
imply that $\tau((a_1, k_1) \cdot (a_2, k_2)) = \tau(a_1, k_1)\tau(a_2, k_2)$.
Since $A \le \Ann(G)$, the same conclusion holds under the operation `$\circ$'.
Hence  $\tau$ is a brace homomorphism.  Using the fact that $\mathbb{C}^\times$ is a trivial brace, it follows that $G/\Ker(\tau)$ is a trivial brace, which gives  $\Gamma_2(G) \le \Ker( \tau)$ as $\Gamma_2(G)$ is the smallest ideal such that $G/\Gamma_2(G)$ is trivial. Since $\tau|_{A}  = \lambda$, it follows that $\lambda(A \cap \Gamma_2(G)) = 1$.
Thus $\lambda \in (A\cap \Gamma_2(G))^\perp$.

Conversely let $\lambda \in (A\cap \Gamma_2(G)^\perp$.  Notice that $A\Gamma_2(G)$, as a subgroup of  $(G, \cdot)$, is an ideal of $G$ and $A/(A \cap \Gamma_2(G)) \cong A\Gamma_2(G)/\Gamma_2(G)$. Since  $A \cap \Gamma_2(G) \le \Ker(\lambda)$, $\lambda$ factors through a homomorphism $f : A/(A \cap \Gamma_2(G)) \to \mathbb{C}^\times$, which, by the injectivity of $\mathbb{C}^\times$,   extends to a homomorphism $\bar{f} : G/\Gamma_2(G) \to \mathbb{C}^\times$. Let us define
$$\delta : G \to G/\Gamma_2(G) \xrightarrow {\bar{f}} \mathbb{C}^\times,$$
the composition of the natural projection $G \to G/\Gamma_2(G)$ and $\bar{f}$. It is easy to see that $\delta$ is a brace homomorphism. 

Let us now define $h : K \to \mathbb{C}^\times$ by setting, for all $k \in K$,
$$h(k) = \delta(1, k).$$
For $k_1, k_2 \in K$, we have
\begin{eqnarray*}
h(k_1)  h( k_2) &=& \delta(1, k_1) \delta(1, k_2)\\
&=& \delta((1, k_1) \cdot (1, k_2))\\
&=& \delta(\alpha(k_1, k_2), k_1 \cdot k_2)\\
&=& \delta(\alpha(k_1, k_2), 1) \delta(1, k_1 \cdot k_2)\\
&=& \lambda(\alpha(k_1, k_2)) h(k_1 \cdot k_2),
\end{eqnarray*}
which implies that 
$$\alpha'(k_1, k_2) = \lambda(\alpha(k_1, k_2)) =  h(k_1)  h( k_2) h(k_1 \cdot k_2)^{-1}.$$
Similarly one obtains
$$\mu'(k_1, k_2) =  \lambda(\mu(k_1, k_2)) = h(k_1)  h( k_2) h(k_1 \circ k_2)^{-1}$$
for all $k_1, k_2 \in K$. Hence, $(\alpha', \mu') \in \B_b^2(K, \mathbb{C}^{\times})$, and therefore it follows that $\Ker(\varepsilon_G) = (A \cap \Gamma_2(G))^{\perp}$.

Again using the fact  $\mathbb{C}^\times$ is injective, we deduce that each homomorphism $A \cap \Gamma_2(G) \to \mathbb{C}^{\times}$ extends to a homomorphism $A  \to \mathbb{C}^{\times}$, proving that the map $\lambda \to \lambda|_{A \cap \Gamma_2(G)}$ from 
$\hat{A}$ to $\widehat{A \cap \Gamma_2(G)}$ is a surjective group homomorphism. Hence we get
$$\varepsilon_G(\hat{A}) \cong \hat{A}/ (A\cap \Gamma_2(G))^\perp \cong \widehat{A\cap \Gamma_2(G)},$$ 
which completes the proof.
\end{proof}

We now prove Theorem A, which generalizes the result of \cite{LV24} on Schur covers: two Schur covers  of a finite skew brace are isoclinic. 

\begin{thm}\label{thm:gener:Schur:cov}
Let $K$ be a finite skew brace and $G = A \times _{(\alpha, \mu)} K$ and $H = B \times _{(\beta, \nu)} K$  two brace extensions of $K$ by $A$ and $K$ by $B$ corresponding to the cohomology classes $[(\alpha, \mu)] \in \Ha^2_b(K, A)$ and  $[(\beta, \nu)] \in \Ha^2_b(K, B)$, respectively.  Let $\xi :  G/A  \to H/B$ be a brace isomorphism with $\xi \big(A (1,k)\big)=\big(B (1, k)\big)$ for all $k \in K$. Then the following are equivalent:
\begin{enumerate}
\item There exists an isomorphism $\theta: \Gamma_2(G) \to \Gamma_2(H)$ such that the diagram
\begin{equation}\label{dia:AB}
\begin{tikzcd}
\Gamma_2(G) \arrow[d,"\theta"] &(G/{A})^2\arrow[l,"\phi_{\cdot}^G"']\arrow[r,"\phi_\ast^G"]\arrow[d,"\xi \times \xi"] &\Gamma_2(G)\arrow[d,"\theta"]\\
\Gamma_2(H) &(H/{B})^2\arrow[l,"\phi_{\cdot}^H"]\arrow[r,"\phi_\ast^H"'] & \Gamma_2(G)
\end{tikzcd}
\end{equation}  commutes.
\item $\varepsilon_G(\hat{A})=\varepsilon_G(\hat{B})$.
\end{enumerate}
\end{thm}

\begin{proof}
Let $L = \big((A \times B) \times_{(\chi, \zeta)} K\big)$ be the extension of $K$ by $A \times B$ corresponding to the cohomology class $[(\chi, \zeta)] \in \Ha^2_b(K, A \times B)$ given by
\begin{align*}
\chi(k_1,k_2) &:=\big(\alpha(k_1,k_2),\beta(k_1,k_2)^{-1}\big)\\
\zeta(k_1,k_2) &:=\big(\mu(k_1,k_2),\nu(k_1,k_2)^{-1}\big).
\end{align*}
To make a differentiation between elements of $G$, $H$ and $L$, we shall write $(a, k)_G$, $(b, k)_H$ and $((a, b), k)_L$ respectively, where $a, b \in A$ and $k \in K$. For simplicity of notation, the element $((1,1), k)_L$ will  be denoted by $(1, k)_L$. As is evident from the definition of $\chi$ and $\zeta$, computations in $L$ are determined by the corresponding computations in $G$ and $H$. For example, the product 
\begin{eqnarray*}
\big((a_1, b_1), k_1)_L \cdot (a_2, b_2), k_2)_L \big) &=& \big((a_1, b_1)(a_2, b_2)\chi(k_1, k_2), k_1 \cdot k_2\big)_L\\
&=&  \big((a_1a_2 \alpha(k_1, k_2), b_1b_2\beta(k_1, k_2)^{-1}), k_1 \cdot k_2\big)_L
\end{eqnarray*}
in $L$ is determined by the products $(a_1, k_1)_G \cdot (a_2, k_2)_G$ and $(b_1, k_1)_H \cdot (b_2, k_2)_H$ in $G$ and $H$ respectively. The same happens with respect to the operation `$\circ$' using $\zeta$.  

We start by noticing that the generating elements of $\Gamma_2(Y)$, $Y = G, H, L$, are determined by the elements $(1, k)_Y$, $k \in K$. We explain it for the elements of $\gamma_2(Y)$ and leave the verifications on the elements of $Y \ast Y$ for the reader. Let $(a, k) \in \gamma_2(G)$. Then there exist elements $(a_i, k_i)$, $1 \le i \le m$, of $G$ such that 
$$(a, k) = \Pi ~[(a_i, k_i), (a_j, k_j)]^{t_{ij}} ~\mbox{ (product in $G$)}.$$
Since $[(a_i, k_i), (a_j, k_j)] = [(1, k_i), (1, k_j)]$, the assertion holds. So it is clear that the element $a$ in  $(a, k) \in \Gamma_2(G)$ is 
a product of elements of the form $\alpha(\Pi k_i, \Pi k_j)$ in $A$. Say $a = \Pi \; \alpha (\Pi k_i, \Pi k_j)$, where $\Pi k_i$ denotes the product in $K$ as a skew brace, that is, involving operations `$\cdot$' and `$\circ$'.  Then, for 
$b:= \Pi \; \beta (\Pi k_i, \Pi k_j)$, it follows that $(b, k) \in \Gamma_2(H)$ and $\big((a, b^{-1}), k\big) \in \Gamma_2(L)$. We can do the same by starting with $(b, k) \in \Gamma_2(H)$, that is, we get $a \in A$ such that  $(a, k) \in \Gamma_2(G)$ and $\big((a, b^{-1}), k\big) \in \Gamma_2(L)$.  Also $\big((a, b^{-1}), k\big) \in \Gamma_2(L)$ implies that $(a, k) \in \Gamma_2(G)$ and $(b, k) \in \Gamma_2(H)$. Just to remark that arbitrary elements $(a, k) \in \Gamma_2(G)$ and $(b, k) \in \Gamma_2(H)$ would not, in general, give $\big((a, b^{-1}), k\big) \in \Gamma_2(L)$.

Now we assume that assertion (1) of the statement holds true.  By the very definition of $\xi$ and the commutativity of the diagram \eqref{dia:AB}, we have
\begin{align}
\theta([(1,k_1)_G,(1,k_2)_G]) &=[(1,k_1)_H,(1,k_2)_H]\\
\theta((1,k_1)_G\ast (1,k_2)_G) &=(1,k_1)_H \ast (1,k_2)_H.
\end{align}
This, using the constructions in the preceding paragraph, says that for any $(a, k) \in \Gamma_2(G)$, there exists a unique $b \in B$ such that  $(b,k) \in \Gamma_2(H)$ and
\begin{equation}
\theta(a,k) = (b,k). \label{eq:thetaValue}
\end{equation}
 So for $(a, 1) \in A \cap \Gamma_2(G)$, we get $\theta(a, 1) = (b, 1) \in B \cap \Gamma_2(H)$.
 
Let $\lambda_1 \in \hat{A}$.   Then we can define a group homomorphism ${\bar \lambda} _2 : \theta(A \cap \Gamma_2(G)) \to \mathbb{C}^{\times}$ by setting 
 $${\bar \lambda}_2(\theta(a)) = \lambda_1(a)$$
 for all $a \in A \cap \Gamma_2(G)$. Since  $\mathbb{C}^{\times}$ is injective,  ${\bar \lambda}_2$ extends to $\lambda_2 \in {\hat B}$.
  Set $\lambda := \lambda_1 \times \lambda_2 : A \times B \to \mathbb{C}^{\times}$ with  $\lambda(a,b)=\lambda_1(a)\lambda_2(b)$.  Then $\lambda \in \widehat{A \times B}$. Let $((a, b^{-1}), 1) \in (A \times B) \cap \Gamma_2(L)$. Then as observed above  $(a, 1) \in A \cap \Gamma_2(G)$ and $(b, 1) \in B \cap \Gamma_2(H)$ such that $\theta(a, 1) = (b, 1)$. Thus $\lambda(a, b^{-1})  = 1$ for all $((a, b^{-1}), 1) \in (A \times B) \cap \Gamma_2(L)$. Now using the setting of Lemma \ref{lem:BraceEpsilonKernel}, it follows that $\lambda \in \Ker(\varepsilon_L)$. Thus $\varepsilon_L(\lambda)  \in \B_b^2(K, \mathbb{C}^{\times})$. For $k_1, k_2 \in K$,  we have 
\begin{align*}
\lambda (\chi(k_1,k_2)) &=\lambda\big((\alpha(k_1,k_2)),(\beta(k_1,k_2)^{-1})\big)\\
&=\lambda_1(\alpha(k_1,k_2)) \lambda_2(\beta(k_1,k_2))^{-1}\\
&=  \alpha' \beta'^{-1} (k_1, k_2)
\end{align*}
and
\begin{align*}
\lambda(\zeta(k_1,k_2)) &= \lambda \big((\mu(k_1,k_2)),(\nu(k_1,k_2)^{-1})\big)\\
&=\lambda_1(\mu(k_1,k_2)) \lambda_2(\nu(k_1,k_2))^{-1}\\
&= \mu' \nu'^{-1}(k_1, k_2).
\end{align*}
This yields $(\alpha' \beta'^{-1}, \mu' \nu'^{-1}) =(\alpha', \mu') (\beta', \nu')^{-1} \in \B_b^2(K, \mathbb{C}^{\times})$, that is, $ \B_b^2(K, \mathbb{C}^{\times}) (\alpha', \mu') = \B_b^2(K, \mathbb{C}^{\times}) (\beta', \nu') $. We have obtained $\varepsilon_G(\hat{A})\subseteq\varepsilon_H(\hat{B})$. Similarly, considering $\xi^{-1}$, $\theta^{-1}$ and the commutativity of the diagram \eqref{dia:AB}, we get $\varepsilon_H(\hat{B}) \subseteq   \varepsilon_G(\hat{A})$, which proves  assertion (2).

Next we assume that assertion (2) holds, that is, $\varepsilon_G(\hat{A})=\varepsilon_G(\hat{B})$.   Since $K$ is finite, by Proposition \ref{gamma2:finite} it follows that  $\Gamma_2(Y)$ is finite for $Y = G, H, L$ (this assertion also follows from \cite[Theorem 5.4]{JKAV21}).  Now, using Lemma \ref{lem:BraceEpsilonKernel}, we get
$$A\cap \Gamma_2(G) \cong \widehat{A\cap \Gamma_2(G)} \cong \varepsilon_G(\hat{A})=\varepsilon_H(\hat{B})\cong\widehat{B\cap \Gamma_2(H)}\cong B\cap \Gamma_2(H).$$

Let $\bar{\lambda}_1 \in \widehat{A\cap \Gamma_2(G)}$. Then  $\bar{\lambda}_1$ extends to $\lambda_1 \in \hat{A}$. By Lemma \ref{lem:BraceEpsilonKernel} again, the mapping 
\begin{equation*}
\tau_1 : \widehat{A\cap \Gamma_2(G)} \to \varepsilon_G(\hat{A}),\hspace{1cm} \tau_1(\bar{\lambda}_1) = \varepsilon_G(\lambda_1) = \B_b^2(K, \mathbb{C}^{\times})  (\alpha', \mu')
\end{equation*}
is an isomorphism. Similarly we have an isomorphism 
\begin{equation*}
\tau_2 : \widehat{B \cap \Gamma_2(H)} \to \varepsilon_H(\hat{B}),\hspace{1cm} \tau_2(\bar{\lambda}_2)=\varepsilon_H(\lambda_2) =\B_b^2(K, \mathbb{C}^{\times})  (\beta', \nu'),
\end{equation*}
where $\lambda_2$ extends $\bar{\lambda}_2$ from $\widehat{B \cap \Gamma_2(H)}$ to $\hat{B}$.
So we  have an isomorphism 
\begin{equation*}
\psi:=\tau_2^{-1}\tau_1:\widehat{A\cap \Gamma_2(G)}\to \widehat{B\cap \Gamma_2(H)},\hspace{1cm} \psi(\bar{\lambda}_1)= \bar{\lambda}_2.
\end{equation*}
Taking into account the natural isomorphism $A\cap \Gamma_2(Y) \cong \widehat{A\cap \Gamma_2(Y)}$, $Y = G, H$,  we get an isomorphism $\eta:A\cap \Gamma_2(G) \to B \cap \Gamma_2(H)$ satisfying 
\begin{equation}
\bar{\lambda}_1(a)=\psi(\bar{\lambda}_1)(\eta(a))  \label{eq:eta}
\end{equation}
for all $\lambda \in \widehat{A\cap \Gamma_2(G)}$  and $a \in A \cap \Gamma_2(G)$. We now aim to extend $\eta$ to a brace isomorphism between $\Gamma_2(G)$ and $\Gamma_2(H)$.

Let $\lambda:=\lambda_1\times \lambda_2\in \widehat{A\times B}$. It follows from the definitions of  $\varepsilon_Y$, $Y = G, H, L$, that $\varepsilon_L(\lambda)$ is a coboundary if and only if  $\varepsilon_G(\lambda_1) =\varepsilon_H(\lambda_2)$, which is equivalent to $\psi(\bar{\lambda}_1) = \bar{\lambda}_2$, where $\bar{\lambda}_1 = \lambda_1 |_{A \cap \Gamma_2(G)}$ and  $\bar{\lambda}_2 = \lambda_2 |_{B \cap \Gamma_2(H)}$. Using \eqref{eq:eta}, the preceding condition is further equivalent to $\bar{\lambda}_1(a) = \bar{\lambda}_2(\eta(a))$ for all $a \in A \cap \Gamma_2(G)$.  Hence, we obtain 
\begin{equation*}
\Ker(\varepsilon_L) = \{\lambda\in\widehat{A \times B} \mid \lambda(a,\eta(a)^{-1})=1~\mbox{ for all } a\in A\cap \Gamma_2(G)\}.
\end{equation*}
By Lemma \ref{lem:BraceEpsilonKernel} we know that
$$\Ker(\varepsilon_L) =  \big((A \times B) \cap \Gamma_2(L)\big)^{\perp}.$$
Reading it in the present set-up, we get
\begin{equation}
(A\times B)\cap \Gamma_2(L)= \{\big((a,\eta(a)^{-1}),1\big) \mid a\in A\cap \Gamma_2(G)\}.\label{eq:ABL'}
\end{equation}

As we already know, for any given $(a, k) \in \Gamma_2(G)$, there exists $(b, k) \in \Gamma_2(H)$ such that $((a,b^{-1}),k) \in \Gamma_2(L)$. Using \eqref{eq:ABL'} we see that the element $(b, k) \in \Gamma_2(H)$ is unique. For, if $(b_1, k) \in \Gamma_2(H)$ is another such  element, that is,  $((a,b_1^{-1}), k) \in \Gamma_2(L)$, then $((a,b^{-1}), k) \cdot  ((a,b_1^{-1}), k)^{-1} \in \Gamma_2(L)$. So $((1, b^{-1}b_1), 1) \in  (A \times B) \cap \Gamma_2(L)$, which, by \eqref{eq:ABL'}, gives $b = b_1$. We can, similarly, say that  for any given $(b, k) \in \Gamma_2(H)$, there exists a unique $(a, k) \in \Gamma_2(G)$ such that $((a,b^{-1}), k) \in \Gamma_2(L)$. This allow us to define a bijective map $\theta : \Gamma_2(G) \to \Gamma_2(H)$, given by 
$$\theta(a, k) = (b, k).$$
Recall that the elements $a \in A$ and $b \in B$ under discussion are given by the same product of  elements of the form $(1, k)$ with respect to the cohomology classes $[(\alpha, \mu)] \in \Ha^2_b(K, A)$ and $[(\beta, \nu)] \in \Ha^2_b(K, B)$ respectively. An easy computation, taking into consideration the definition of $L$, shows that $\theta$ is a brace isomorphism, and the diagram \eqref{dia:AB} commutes. This completes the proof of the theorem. 
\end{proof}

If $K$ is  a trivial skew brace, then the preceding theorem provides a modified proof of \cite[Theorem 1.7]{Tappe76}. This statement requires a justification. Let $G = A \times_r K$ be the group  extension of a group $K$ by an abelian group $A$ (acted upon trivially by $K$) with respect to the $2$-cocycle $r$ of $K$ with coefficients in $A$.  We remark that the existing proof of  \cite[Theorem 1.7]{Tappe76} uses the fact that for any $k \in \gamma_2(K)$, the  element  $(1, k)$ can be  viewed as an element of  $\gamma_2(G)$ by suitably modifying the $2$-cocycle $r$. Unfortunately, we could not see how to get this. So we found a slightly different proof without using this argument.

%%%%%%%%%%%%%%%%%%%%%%%%%%%%%%%%%%%%%%%%%%%%%%%%%%%%%%%%%%%%%

 \section{Skeleton of the Lower Central Series}
 
Let $A$ be a skew  brace. We now carry out a construction motivated by the similar concept in \cite[page 16]{CSV19}. Let $M$ denote the free monoid generated by the set of symbols $X := \{\gamma_{.}, \bar{\gamma}_{.}, *, \bar{*}\}$, where, for $a, b \in A$,
\begin{center}
\[\gamma_{.}(a, b) = [a, b]^{.}, \;\; \bar{\gamma}_{.}(a, b) = [b, a]^{.}, \;\; *(a, b) = a*b,\;\;  \bar{*}(a, b) = b*a.\]
\end{center}
Let $m \in M$. Then 
$$m =   \epsilon_1 \epsilon_2 \cdots \epsilon_s,$$
where $\epsilon_i \in X$, $1 \le i \le s$. The positive integer $s$, in the above expression,  is said to be the degree of $m$, denoted as $\deg(m) = s$. Let $A_i$, 1$ \le i \le s+1$, be  subgroups of $(A, \cdot)$. We now define the (brace commutator) word map $m : A_1 \times A_2 \times \cdots  \times A_{s+1} \to A$ by setting
$$m(a_1 a_2 \cdots a_{s+1}) :=  \epsilon_1(a_1, \epsilon_2(a_2, \ldots, \epsilon_s(a_s, a_{s+1}) \cdots )),$$
where $a_i \in A_i$, $1 \le i \le s+1$. The set of all word values of $m$ is denoted by $m(A_1 A_2 \cdots A_{s+1})$. By $m(A_1 A_2 \cdots A_{s+1})^{-1}$ we denote the set  $\{a \in A \mid a^{-1} \in m(A_1 A_2 \cdots A_{s+1})\}$.

Whenever $A_i = A$, $1 \le i \le s$, we denote by $A(s)$ the cartesian product $A_1 \times A_2 \times \cdots  \times A_{s}$. We define degree-lexicographic order on $M$ by setting $* < \bar{*} < \gamma_. < \bar{\gamma}_.$. For  the elements $m_2 < m_1$ of $M$, we have 
\begin{eqnarray*}
m_2(a_1 a_2 \cdots a_{\deg(m_2)+1})   m_1(b_1 b_2 \cdots b_{\deg(m_1)+1}) &=& m'\big(m_2(a_1 a_2 \cdots a_{\deg(m_2)+1}) b_1 b_2 \cdots b_{\deg(m_1)+1} \big)\\
 & & m_1(b_1 b_2 \cdots b_{\deg(m_1)+1})  m_2(a_1 a_2 \cdots a_{\deg(m_2)+1}),
\end{eqnarray*}
where $m' = \gamma_.m_1$ with $\deg(m') = \deg(m_1) +1$. As a result,  for a non-negative integer $t$,  it follows that the elements of the additive subgroup (say) $A_{(t)}$ of $A$ generated by 
$$\{m(A(\deg(m)+1)) \mid m \in M \mbox{ and } \deg(m) \ge t\}$$
can be written as $a_1a_2 \cdots a_s$ for some positive integer $s$, where $a_i \in  m_i(A(\deg(m_i)+1)) \cup m_i(A(\deg(m_i)+1))^{-1}$, $m_1 \ge m_2 \ge \cdots \ge m_s$ and $\deg(m_s) \ge t$.  It turns out that $A_{(t)}$ is a strong left  ideal of $A$ contained in $\Gamma_{t+1}(A)$. We now prove

\begin{lemma}\label{lemma-key}
Let  $u$ be an element of a skew brace $A$.  Then  $u \in \Ann_t(A)$ if and only if $m(a_1 a_2 \cdots a_t u) = 1$ for all $m \in M$ with $\deg(m) = t$ and $a_i \in A$, $1 \le i \le t$.
\end{lemma}
\begin{proof}
We proceed by induction on $t$. Let $t=1$. Then $m \in X$. The result now follows from the definition of $\Ann_1(A)$. Assume by induction that the result is true for all $1 \le k < t$. Let $u \in \Ann_{t}(A)$ and $m = \epsilon_1  \epsilon_2 \cdots \epsilon_t \in M$. Note that $\epsilon_t \in X$. In any case, by the definition, it follows that  $\epsilon_t(a_t, u) \in \Ann_{t-1}(A)$. Now, for any $a_1, a_2, \ldots, a_t \in A$, we have
\begin{eqnarray*}
m(a_1a_2 \cdots a_t u)  &=&  \epsilon_1(a_1, \epsilon_2(a_2, \ldots, \epsilon_t(a_t, u) \cdots ))\\
&=& \epsilon_1(a_1, \epsilon_2(a_2, \ldots, \epsilon_{t-1}(a_{t-1}, v) \cdots )), \mbox{ where } v= \epsilon_t(a_t, u),\\
&=& 1 \mbox{ (using inductive argument), since } v \in \Ann_{t-1}(A).
\end{eqnarray*}
Now assume the other side, that is,  $m(a_1 a_2 \cdots a_t u) = 1$ for all $m \in M$ with $\deg(m) = t$ and $a_i \in A$, $1 \le i \le t$.  Let $m' \in M$ with $\deg(m') = t-1$ and $\epsilon_t \in X$. Then $\deg(m' \epsilon_t) = t$, and therefore, by the given hypothesis, $(m' \epsilon_t )(a_1 a_2 \cdots a_t u) = 1$, which gives  $m' (a_1 a_2 \cdots a_{t-1}v_t) = 1$ for all $a_i \in A$, $1 \le i \le t-1$, where $v_t :=\epsilon_t(a_t, u)$. Since $m' \in M$ is an arbitrary element with $\deg(m') = t-1$, by inductive argument, we have $v_t \in \Ann_{t-1}(A)$. Further, since   $\epsilon_t \in X$ and  $a_t \in A$ are arbitrary, it follows that $u \in \Ann_t(A)$.  The proof is complete. 
\end{proof}

By \cite[Theorem 2.8]{BJ23} and the preceding result, we get
\begin{thm}
For any skew brace $A$ and an integer $t \ge 1$, $A_{(t)} = 1$ if and only if $\Gamma_{t+1}(A) = 1$.
\end{thm}

So we get a decreasing series $A \ge A_{(1)} \ge A_{(2)} \ge \cdots$   of strong  left ideals of $A$ which determines the length of the lower central series $A = \Gamma_1(A) \ge \Gamma_2(A) \ge \cdots$. For this reason, we call the series $A \ge A_{(1)} \ge A_{(2)} \ge \cdots$ the {\em skeleton} of the lower central series $A = \Gamma_1(A) \ge \Gamma_2(A) \ge \cdots$.

We now consider another series of strong left ideals of a skew brace $A$. Set $\bar{\Gamma}_1(A) = A$, and inductively define 
$$\bar{\Gamma}_{n}(A) = \langle A * \bar{\Gamma}_{n-1}(A), [A, \bar{\Gamma}_{n-1}(A)]^{\cdot}\rangle^{\cdot}$$
for all $n \ge 2$. Note that $\bar{\Gamma}_2(A) = \Gamma_2(A)$. It turns out that $\bar{\Gamma}_{n}(A)$ is a strong left ideals of $A$ for all $n \ge 1$, and we get a decreasing series $A = \bar{\Gamma}_1(A)  \ge \bar{\Gamma}_2(A)   \ge \cdots$ of  strong left ideal  of $A$. We now prove

\begin{lemma}\label{lem:inclusion:1}
For any skew brace $A$ and $n \ge 1$, $\bar{\Gamma}_{n}(A) \subseteq A_{(n-1)} \subseteq \Gamma_n(A)$.
\end{lemma}
\begin{proof}
We only need to prove the first containment. We proceed by induction. The result follows for $n = 1, 2$. So assume that it follows up to $n-1$, $n \ge 3$, that is, $\bar{\Gamma}_{n-1}(A) \subseteq A_{(n-2)}$. So any $v \in  \bar{\Gamma}_{n-1}(A)$ can be written in the form $w_1w_2 \cdots w_s$ for some positive integer $s$, where $w_i \in  m_i(A(\deg(m_i)+1)) \cup m_i(A(\deg(m_i)+1))^{-1}$ and $\deg(m_s) \ge n-2$. Then, an easy computation shows that  $\epsilon(a, w_1w_2 \cdots w_s) \in A_{(n-1)}$ for $\epsilon = *, \gamma_{\cdot}$, because $A_{(n-1)}$ is normal in $(A, \cdot)$.  Indeed, one may like to use the identity $[a, w_i^{-1}]^{\cdot} = \big(([a, w_i]^{\cdot})^{-1}\big)^{w_i^{-1}}$. The proof is complete.
\end{proof}

\begin{lemma}\label{lemma4}
For any skew brace $A$ and for all $s \ge r \ge 1$,   $u \in \Ann_s(A)$ and $v \in \bar{\Gamma}_r(A)$,  we have $\epsilon(u, v) \in  \Ann_{s-r}(A)$, where $\epsilon = \gamma_{\cdot}, *$.
\end{lemma}

\begin{proof}
By the definition of the $s$th annihilator, the result holds for $r=1$. We now proceed by induction. Assume that for all $s \ge r > 1$,  $v \in \bar{\Gamma}_{r-1}(A)$ and $u \in \Ann_s(A)$,  we have $\epsilon(u, v) \in  \Ann_{s-r+1}(A)$, where $\epsilon \in \{\gamma_{\cdot}, * \}$. It is sufficient to prove the assertion for the generators of $\bar{\Gamma}_r(A)$. Let $v = \epsilon_1(a, c)$, where $a \in A$ and $c \in \bar{\Gamma}_{r-1}(A)$ and $\epsilon_1 \in  \{\gamma_{\cdot}, * \}$.  By our inductive assumption we know that $\epsilon_1(u, c)  \in \Ann_{s-r+1}(A)$. We first assume that $\epsilon = \epsilon_1 = \gamma_{\cdot}$.   Then using Hall-Witt identity in $(A, \cdot)$, we get
$$([u, [a, c]^{\cdot}]^{\cdot})^{a^{-1}} ([c, [u, a]^{\cdot}]^{\cdot})^{u^{-1}} ([a, [c, u]^{\cdot}]^{\cdot})^{c^{-1}} =  1.$$
Now using the fact that $[u, a]^{\cdot} \in \Ann_{s-1}(A)$, by the inductive argument it follows that 
$$([c, [u, a]^{\cdot}]^{\cdot})^{u^{-1}} ([a, [c, u]^{\cdot}]^{\cdot})^{c^{-1}} \in \Ann_{s-r}(A).$$
 Hence $[u,v]^{\cdot} \in \Ann_{s-r}(A)$, in this case.

 Next, we take $\epsilon = \gamma_{\cdot}$ and $\epsilon_1 = *$. Set $x :=  [\lambda_{a}(c), [u, c^{-1}]^{\cdot}]^{\cdot}$ and $y :=  [c, [u, c^{-1}]^{\cdot}]^{\cdot}$ which lie in $\Ann_{s-r}(A)$ by the assumption. Then 
\begin{eqnarray*}
[u, a * c]^{\cdot} &=& [u, \lambda_{a}(c)c^{-1}]^{\cdot} \\
&=& [u, \lambda_{a}(c)]^{\cdot} [\lambda_{a}(c), [u, c^{-1}]^{\cdot}]^{\cdot} [u, c^{-1}]^{\cdot}\\
&=& \lambda_{a}([\lambda^{-1}_{a}(u), c]^{\cdot}) x ([c, [u, c^{-1}]^{\cdot}]^{\cdot})^{-1} ([u, c]^{\cdot})^{-1}\\
&=& \lambda_a([wu, c]^{\cdot}) xy ([u, c]^{\cdot})^{-1}, \mbox{ where $w \in \Ann_{s-1}(A)$ } \\
&=& \lambda_a\big([w, [u, c]^{\cdot}]^{\cdot} [u, c]^{\cdot} [w, c]^{\cdot}\big) xy ([u, c]^{\cdot})^{-1}\\
&=&  \lambda_a([w, [u, c]^{\cdot}]^{\cdot}) \lambda_a([u, c]^{\cdot}) \lambda_a ([w, c]^{\cdot}) xy ([u, c]^{\cdot})^{-1}\\
&=&  \lambda_a([w, [u, c]^{\cdot}]^{\cdot})  (\lambda_a ([w, c]^{\cdot}) xy )^{\lambda_a([u, c]^{\cdot})} \lambda_a([u, c]^{\cdot})([u, c]^{\cdot})^{-1}\\
&=&  \lambda_a([w, [u, c]^{\cdot}]^{\cdot})  (\lambda_a ([w, d]^{\cdot}) xy )^{\lambda_a([u, c]^{\cdot})} (a * [u, c]^{\cdot}),
\end{eqnarray*}
which lies in $\Ann_{s-r}(A)$, since $[w, [u, c]^{\cdot}]^{\cdot}, [w, c]^{\cdot}$  both lie in  $\Ann_{s-r}(A)$ by inductive argument, and $a * [u, c]^{\cdot} \in \Ann_{s-r}(A)$ by the definition of $(s-r-1)$th annihilator. This proves that $[u,v]^{\cdot} \in \Ann_{s-r}(A)$.
 
 Now we consider $\epsilon = \epsilon_1 =*$. Then
\begin{eqnarray*}
\lambda_u(a * c) &=& \lambda_u\big(\lambda_{a}(c) c^{-1}\big)\\
&=&  \lambda_{u \circ a} (c) \lambda_u( c^{-1})\\
&=&  \lambda_{a \circ u \circ [\bar{u}, \; \bar{a}]^{\circ}} (c) \lambda_u( c)^{-1}\\
&=&  \lambda_{a \circ u}  (wc) \lambda_u( c)^{-1}, \mbox{ where $w \in \Ann_{s-r}(A)$ } \\
&=&  \lambda_{a \circ u}(w)  \lambda_a (\lambda_u( c) c^{-1} c ) c^{-1} (\lambda_u( c)c^{-1})^{-1} \\
&=&  \lambda_{a \circ u}(w)  \lambda_a (x) (a * c)  x^{-1}, \mbox{ where $x \in \Ann_{s-r+1}(A)$ }\\
&=& \lambda_{a \circ u}(w)  \lambda_a (x) x^{-1} [x, (a * c)]^{\cdot} (a*c)\\
&=& d (a * c), 
\end{eqnarray*}
where $d =  \lambda_{a \circ u}(w)  \lambda_a (x) x^{-1} [x, (a * c)]^{\cdot} \in \Ann_{s-r}(A)$. 

Finally we consider  $\epsilon = *$ and $\epsilon_1  = \gamma_.$. Then
\begin{eqnarray*}
\lambda_u([a, c]^{\cdot}) &=& [\lambda_u(a), \lambda_u(c) c^{-1} c]^{\cdot}\\
&=& [\lambda_u(a), \lambda_u(c) c^{-1}]^{\cdot}  [\lambda_u(c) c^{-1}, [\lambda_u(a),  c]^{\cdot}]^{\cdot} [\lambda_u(a), c]^{\cdot}\\
&=& x_1y_1 [\lambda_u(a)a^{-1} a, c]^{\cdot}, \mbox{ where $x_1, y_1 \in \Ann_{s-r}(A)$ } \\
&=& x_1y_1 [\lambda_u(a)a^{-1}, [a, c]^{\cdot}]^{\cdot} [a, c]^{\cdot} [\lambda_u(a)a^{-1}, c]^{\cdot}\\
&=& x_1y_1 [\lambda_u(a)a^{-1}, [a, c]^{\cdot}]^{\cdot} ([\lambda_u(a)a^{-1}, c]^{\cdot})^{[a, c]^{\cdot} } [a, c]^{\cdot} \\
&=& d [a, c]^{\cdot},
\end{eqnarray*}
where $d =  x_1y_1 [\lambda_u(a)a^{-1}, [a, c]^{\cdot}]^{\cdot} ([\lambda_u(a)a^{-1}, c]^{\cdot})^{[a, c]^{\cdot} } \in \Ann_{s-r}(A)$. This completes the proof.  
\end{proof}

For any $m \in M$ with $\deg(m)=1$, it is proved by Letourmy and Vendramin \cite{LV23} that maps $\phi^A_m :  (A/\Ann_n(A))^{2} \to \Gamma_{n}(A)$, given by 
$$\phi_m^A(\tilde{a}_1,  \tilde{a}_2)  = m(a_1 a_2),$$
are well defined. This paved a way for them to define isoclinism of skew left  braces, defined in Section 2.  In group theory we have a concept  of $n$-isoclinism  of groups, where  isoclinism is $1$-isoclinism. We will prove below that it is not possible to extend the concept of isoclinism to $n$-isoclinism of skew left braces with the current setting of lower central series and annihilator series. However, it can be established in some specific, but still large, class of skew braces. We now make preparations to do so.

\begin{lemma}\label{lemma4b}
Let $A$ be a skew brace and $\epsilon \in X$ be such that  $\epsilon(c, d) \in \Ann_{s-r}(A)$  for all  $c \in \Ann_s(A)$ and $d \in \Gamma_r(A)$, where $s \ge r \ge 1$. Then, for $a \in A$,  $u \in \Ann_s(A)$,  $v \in \Gamma_r(A)$ and $w \in \Ann_{s-r+1}(A)$,  we get $\epsilon(au, vw) = \epsilon(a, v)$ modulo $\Ann_{s-r}(A)$.
\end{lemma}

\begin{proof}
The assertion holds for $\epsilon = \gamma_., \bar{\gamma}_.$ by group theoretic arguments. So let $\epsilon = *$. Then 
\begin{eqnarray*}
(au)*(vw) &=& \lambda_{au}(vw)(vw)^{-1} = \lambda_{au}(v) \lambda_{au}(w) w^{-1}v^{-1}\\
&=&  \lambda_a(x*v) \lambda_a(v) v^{-1} v \big((au)*w\big) v^{-1}, \mbox{   where $x := \lambda^{-1}_a(u)$}\\
&=&  \lambda_a(x*v) (a*v) v \big((au)*w\big)v^{-1},
\end{eqnarray*}
where both $\lambda_a(x*v)$ and  $v \big((au)*w\big)v^{-1}$ lie in $\Ann_{s-r}(A)$. Finally let  $\epsilon = \bar{*}$. Then 
\begin{eqnarray*}
(vw)*(au) &=& \lambda_{vw}(a)  \lambda_{vw}(u) u^{-1} a^{-1} \\
&=& \lambda_v(\lambda^{-1}_v(w)*a) (v*a) a  \lambda_v(\lambda^{-1}_v(w)*u) (v*u) a^{-1},
\end{eqnarray*}
where  both $\lambda_v(\lambda^{-1}_v(w)*a)$ and $a  \lambda_v(\lambda^{-1}_v(w)*u) (v*u) a^{-1}$ belong to $\Ann_{s-r}(A)$. The proof is complete.
\end{proof}

Set $S_0(A)  := A$ and, for any integer $r \ge 1$, 
$$S_r(A) := \{m(a_1a_2 \cdots a_{r+1}) \mid m \in M,  \deg(m) = r  \mbox{ and } a_1,  \ldots, a_{r+1} \in A\}.$$
 Now specialising to $S_r(A)$  in  the preceding result, we get

\begin{cor}\label{cor-phi}
Let $A$ be a skew brace and $\epsilon \in X$ be such that  $\epsilon(c, d) \in \Ann_{s-r}(A)$  for all  $c \in \Ann_s(A)$ and $d \in S_{r-1}(A)$, where $s \ge r \ge 1$. Then, for $a \in A$,  $u \in \Ann_s(A)$,  $v \in S_{r-1}(A)$ and $w \in \Ann_{s-r+1}(A)$,  we get $\epsilon(au, vw) = \epsilon(a, v)$ modulo $\Ann_{s-r}(A)$.
\end{cor}

Let $\mathcal{I}_n$ denote the class of all skew  braces $A$ for which the following maps are well defined for all $m \in M$ with $\deg(m)= n$:
\begin{equation}\label{map:phi}
\phi_m^A: (A/{\Ann_n(A)})^{n+1}\to A_{(n)} \mbox{ given by } \phi_m^A(\tilde{a}_1 \tilde{a}_2 \cdots \tilde{a}_{n+1}) = m(a_1a_2\cdots a_{n+1}),
\end{equation}
where $\tilde{a}_i =  \Ann_n(A) a_i$ for $i = 1, 2, \ldots, n+1$.
Denote by $\bar{\mathcal{I}}_n$ the subclass of $\mathcal{I}_n$ consisting of skew braces $A$ such that $A$ belongs to $\mathcal{I}_r$ for all $1 \le r \le  n$. The maximal subclass of $\bar{\mathcal{I}}_n$ which is closed under the formation of sub-skew braces, brace homomorphic images and finite direct products will be denoted by $\tilde{\mathcal{I}}_n$. Obviously  $\tilde{\mathcal{I}}_n \subseteq  \bar{\mathcal{I}}_n \subseteq \mathcal{I}_n$ for all integers $n \ge 1$. Note that $ \bar{\mathcal{I}}_2 = \mathcal{I}_2$. With this setup, we now derive an equivalent criterion on a skew brace $A$ so that $\phi_m^A$ is well defined. First we derive such an criterion for the class $\bar{\mathcal{I}}_n$.

\begin{thm}\label{phi-well-defined}
A skew brace $A$ belongs to $\bar{\mathcal{I}}_n$ if and only if  $\epsilon(u, v) \in \Ann_{s-r}(A)$ for all $\epsilon \in X$, $u \in \Ann_s(A)$ and $v \in S_{r-1}(A)$, $1 \le r  \le s \le n$. 
\end{thm}

\begin{proof}
For the proof of the `if' part, we only need to prove that 
 $$m(a_1u_1, a_2u_2, \ldots, a_{t+1}u_{t+1}) = m(a_1, a_2, \ldots, a_{t+1})$$
 for all $m \in M$ with $\deg(m) = t$,  $a_i \in A$ and $u_i \in \Ann_{t}(A)$, where $1 \le  i \le t+1$ and $1 \le t \le n$. We only prove the case $t = n$. All other cases, indeed, will follow from the steps of the proof of this case.   Let $m = \epsilon_1 \epsilon_2 \ldots \epsilon_n$, where $\epsilon_i \in X$, be an arbitrary element of $M$ of degree $n$. It follows from Corollary \ref{cor-phi}, for $r = 1$,   that 
$$\epsilon_n(a_nu_n, a_{n+1}u_{n+1}) = \epsilon_n(a_n, a_{n+1})z_{n-1}$$
 for some $z_{n-1} \in \Ann_{n-1}(A)$.  Since $\epsilon_n(a_n, a_{n+1}) \in S_1(A)$,  again by Corollary \ref{cor-phi} for $r =2$, using the given hypothesis, we get 
 $$\epsilon_{n-1}\big(a_{n-1}u_{n-1}, \epsilon_n(a_n, a_{n+1})z_{n-1}\big) = \epsilon_{n-1}\big(a_{n-1}, \epsilon_n(a_n, a_{n+1})\big)z_{n-2}$$
 for some $z_{n-2} \in \Ann_{n-2}(A)$. Set $y:=\epsilon_{n-1}(a_{n-1},\epsilon_n(a_n,a_{n+1}))$; then using the given hypothesis, we get  
 $$\epsilon_2(a_2u_2, \epsilon_3(a_3u_3,  \ldots, \epsilon_{n-2}(a_{n-2}u_{n-2}, yz_{n-2}) \ldots )) = \epsilon_2(a_2, \epsilon_3(a_3, \ldots \epsilon_n(a_n, a_{n+1}) \ldots ))z_1$$
  for some $z_1 \in \Ann_1(A)$. Since $\epsilon_2(a_2, \epsilon_3(a_3, \ldots \epsilon_n(a_n, a_{n+1}) \ldots )) \in S_{n-1}(A)$ and $u_1 \in \Ann_n(A)$, one final iteration gives the desired  identity.
  
 Now we prove the `only if' part.  Assume that $A \in \mathcal{I}_t$, $1 \le t \le n$. For $n \ge s \ge r \ge 1$, let $u \in \Ann_s(A)$ and $v = m'(b_1 b_2 \cdots b_r) \in S_{r-1}$, where  $m' \in M$ with $\deg(m') = r-1$. Choose any $m \in M$ with $\deg(m) = s-r$  and any $\epsilon \in X$. Then, since $A \in \mathcal{I}_s$ and $\deg(m \epsilon m') = s$, we get
\begin{eqnarray*}
m( a_1 \cdots a_{s-r}, \epsilon(u, v) \cdots )&=& m\epsilon m'(a_1 \cdots a_{s-r} u b_1   \cdots b_r)\\
&=&  m\epsilon m'(a_1 \cdots a_{r-1} 1 b_1   \cdots b_r)\\
&=&  1,
\end{eqnarray*}
for all $a_i \in A$, $1 \le i \le s-r$.  Hence, by Lemma \ref{lemma-key}, it follows that $\epsilon(u,v) \in \Ann_{s-r}(A)$.  The proof is now complete.
\end{proof}

Now specializing  to $s = n$, the preceding  theorem provides a criterion for a skew brace to fall in $\mathcal{I}_n$.

\begin{cor}
A skew brace $A$ belongs to $\mathcal{I}_n$ if and only if  $\epsilon(u, v) \in \Ann_{n-r}(A)$ for all $\epsilon \in X$, $u \in \Ann_n(A)$ and $v \in S_{r-1}(A)$, $1 \le r \le n$. 
\end{cor}

Two  skew  braces $A, B \in \mathcal{I}_n$ are said to be \emph{$n$-isoclinic}, $n \ge 1$, if  there exist brace isomorphisms 
\begin{eqnarray*}
\xi &:&  A/\Ann_n(A) \to B/\Ann_n(B)\\
\theta &:& A_{(n)} \to B_{(n)}
\end{eqnarray*}
 such that the following diagram commutes:
 \begin{center}
\begin{tikzcd}\nonumber
(A/\Ann_n(A))^{n+1} \arrow [d, "(\xi)^{n+1}"]   \arrow[r, "\phi_m^A"]  & A_{(n)} \arrow[d, "\theta"]\\
(A/\Ann_n(B))^{n+1}    \arrow[r, "\phi_m^B"]  & B_{(n)}
\end{tikzcd}
\end{center}
for all $m \in M$ with $\deg(m) = n$, where  $(\xi)^{n+1} := \xi \times \cdots \times  \xi \mbox{ (n+1 copies)}$. The pair of  isomorphisms $(\xi, \theta)$ is called an \emph{$n$-isoclinism} between $A$ and $B$. It is easy to see that  $n$-isoclinism is an equivalence relation (reflexive, symmetric and transitive)  in  $\mathcal{I}_n$.

We would like to mention here that the preceding definition generalizes the analogous concept of $n$-isoclinism of finite groups, introduced by  P. Hall \cite{Hall40}, to skew braces. Two groups $G$ and $H$ are said to be \emph{$n$-isoclinic}, $n \ge 1$,  if there exist group isomorphisms
    \begin{eqnarray*}
        \alpha &:& G/\Z_n(G) \longrightarrow H/\Z_n(H), \\
        \beta &:& \gamma_{n+1}(G) \longrightarrow \gamma_{n+1}(H),
    \end{eqnarray*}
    such that the following diagram commutes:
    \begin{center}
    \begin{tikzcd} \label{cd:add_isoclinism}
    (G/\Z_n(G))^{n+1} \arrow [d, "(\alpha)^{n+1}"]  \arrow[r,
"\gamma_{\cdot}^n"]  & \gamma_{n+1}(G) \arrow[d, "\beta"]\\
    (H/\Z_n(H))^{n+1}    \arrow[r, "\gamma_{\cdot}^n"]  &
\gamma_{n+1}(H),
    \end{tikzcd}
    \end{center}
    where $\gamma_{\cdot}^n = \gamma_{\cdot} \cdots \gamma_{\cdot}$ ($n$-copies) is a word of degree $n$ in $M$.
The pair $(\alpha, \beta)$ is called an \textit{$n$-isoclinism} between the groups  $G$ and $H$.

 \vspace{.1in}  
For $n = 2$, we have
\begin{lemma}\label{lemma:2:iso}
A skew  brace $A \in \mathcal{I}_2$ if and only if $u \bar{*} v = 1$ for all $u \in \Ann_2(A)$, $v \in \Gamma_2(A)$.
\end{lemma}
 
 \begin{proof}
 We only need to proof the `if' part of the statement.   So assume that  $u \bar{*} v = 1$ for all $u \in \Ann_2(A)$, $v \in \Gamma_2(A)$. Since $\Gamma_2(A)$ is generated by the elements of the form $a * b$ and $[a, b]^{\cdot}$, the proof follows from Lemma \ref{lemma4}.
\end{proof}

One may very well expect that $\mathcal{I}_2 = \mathcal{SB}$, which, in view of  Lemma \ref{lemma:2:iso}, actually requires to prove that   $u  \bar{*}  v = v * u =1$ for all $u \in \Ann_2(A)$ and $v \in \Gamma_2(A)$.
 But, unfortunately, it fails.  There are plenty of counterexamples to this statement.
A GAP \cite{GAP} search, using the package  Yang-Baxter \cite{VK22}  reveals that this is not true in general. Counterexamples exist even for left braces of order $16$. Here are some GAP ID's of such skew  braces: [ 16, 73 ], [ 16, 74 ], [ 16, 75 ], [ 16, 78 ],  [ 16, 82 ],  [ 16, 85 ], [ 16, 88 ], [ 16, 89 ], [16, 240 ] [ 48, 5011], [48, 5012], [48, 5015 - 5021],  [81, 5174], [81, 5181 ],  [81, 5292 - 5294],  [ 81, 5297 ],  [81, 5298 ], [108, 174], [108, 176], [108, 183], [108, 187], [108, 200], [108, 206], [112, 4969 - 4976]. But the good news is that  there are some considerably large subclasses of skew braces which lie in $\mathcal{I}_n $. One such subclass is the class of all symmetric  skew braces, which we will deal  in the next two sections. One other such subclass consists of $\lambda$-homomorphic skew braces.  Both of these are defined in Section 2.

Skew braces of orders 16, 32, 64, 80, 96, 128  are not implemented in GAP.  As it is evident from the following table that all skew braces having order $n \le 135$, excepting $n = 16, 32, 48, 64, 80, 81, 96, 108, 112, 128$, are in $\mathcal{I}_2$. It is also evident that most of the skew braces of order $n = 16, 48, 81, 108, 112$ are in $\mathcal{I}_2$. More precisely,  48, 800, 345, 32 and 336  skew braces of orders 16, 48, 81, 108 and 112 respectively are not in $\mathcal{I}_2$.

In the following table $m$, $s(m)$, $Sym(m)$ and  $\mathcal{I}_2(m)$, respectively denote  the order of a skew brace,  the number of the skew braces of order $m$, the number of the symmetric skew braces of order $m$ and the number of the skew braces of order $m$ in $\mathcal{I}_2$.
We exclude orders $m$ such that  $m$ and $\phi(m)$ are coprime,  where $\phi$ denotes the Euler’s totient function, because there is a unique skew brace, that is, the trivial skew brace, for such an order m.
 
\begin{table}[H]
\centering
\begin{small}
 \begin{tabular}{|c | c c c c c c c c c c  c c c c|} 
 \hline
 $m$ & $4$ & $6$ & $8$ & $9$ &  $10$ & $12$ & $14$ & $16$ & $18$ &$20$& $21$ & $22$ & $24$ & $25$ \\
 $s(m)$ & $4$ & $6$ & $47$ & $4$ & $6$ & $38$ & $6$ & $1605$ &$49$ &$43$ & $8$ & $6$ & $855$ & $4$ \\
 $Sym(m)$ &  $4$ & $5$ & $39$ & $4$ & $5$ & $26$ & $5$ & $1086$ &$24$ &$28$ & $6$ & $5$ & $501$ & $4$ \\
 $\mathcal{I}_2(m)$ & $ 4$ & $6 $ & $47$ & $4$ & $6$ & $38$ & $6$ & $1557$ &$49$ &$43$ & $8$ & $6$ & $855$ & $4$ \\
 \hline
  $m$ &$26$ & $27$ & $28$ & $30$ & $34$ &  $36$ & $38$ & $39$ & $40$ & $42$ &$44$ & $45$ & $46$ & $48$ \\
 $s(m)$ &$6$ & $101$ & $29$ & $36$ & $6$ & $400$ & $6$ & $8$ & $944$ &$78$ &$29$  & $4$ & $6$ & $66209$ \\
 $Sym(m)$ & $5$ &  $69$ & $21$ & $25$ & $5$ & $156$ & $5$ & $6$ & $514$ &$43$ &$21$ & $4$ & $5$ & $29347$ \\
 $\mathcal{I}_2(m)$ &$6$ & $ 101$ & $ 29$ & $36$ & $6$ & $400$ & $6$ & $8$ & $944$ &$78$ &$29$ & $4$ & $5$ & $65409$ \\
  \hline
 $m$ & $49$ &$50$ & $52$ & $54$ & $55$ & $56$ &  $57$ & $58$ & $60$ & $62$ & $63$ &$66$ & $68$ & $70$ \\ 
 $s(m)$ & $4$ &$51$ & $43$ & $1028$ & $12$ & $815$ & $8$ & $6$ & $418$ & $6$ &$47$ &$36$ & $43$ & $36$ \\
 $Sym(m)$ & $4$ &$24$ &  $28$ & $209$ & $8$ & $492$ & $6$ & $5$ & $190$ & $5$ &$30$ &$25$ & $28$ & $25$ \\
 $\mathcal{I}_2(m)$ & $4$ &$51$ & $43$ & $1028$ & $12$ & $815$ & $8$ & $6$ & $418$ & $6$ &$47$ &$36$ & $43$ & $36$ \\
   \hline
   $m$ & $72$ & $74$ &$75$ & $76$ & $78$ & $81$ & $82$ &  $84$ & $86$ & $88$ & $90$ & $92$ &$93$ & $94$ \\ 
 $s(m)$ & $17790$ & $6$ &$14$ & $29$ & $78$ & $8436$ & $6$ & $606$ & $6$ & $800$ & $294$ &$29$ &$8$ & $6$ \\
 $Sym(m)$ & $5403$ & $5$ &$9$ &  $21$ & $39$ & $4617$ & $5$ & $254$ & $5$ & $483$ & $120$ &$21$ &$6$ & $5$ \\
 $\mathcal{I}_2(m)$ & $17790$ & $6$ &$14$ & $29$ & $78$ & $8091$ & $6$ & $606$ & $6$ & $800$ & $294$ &$29$ &$8$ & $6$\\
   \hline
 $m$ & $98$ & $99$ & $100$ &$102$ & $104$ & $105$ & $106$ & $108$ &  $110$ & $111$ & $112$ & $114$ & $116$ &$117$ \\ 
 $s(m)$ & $53$ & $4$ & $711$ &$36$  & $944$ & $8$ & $6$ & $11223$ & $94$  & $8$ & $8016$ &$78$ &$43$ & $47$ \\
 $Sym(m)$ & $24$ & $4$ & $231$ &$25$ &  $514$ & $6$ & $5$ & $2627$ & $51$ & $6$ & $2533$ & $43$ &$28$ &$30$ \\
 $\mathcal{I}_2(m)$  & $53$ & $4$ & $711$ &$36$ & $944$ & $ 8$ & $6$ & $11191$ & $94$ & $8$ & $7680$ & $78$ &$43$ &$47$ \\
   \hline
   $m$ & $118$ & $120$ & $121$ & $122$ &$124$ & $125$ & $126$ & $129$ & $130$ & $132$ & $134$ & $135$ & $$ & $$\\
   $s(m)$  & $6$ & $22711$ & $4$ &$6$ & $29$ & $213$ & $990$ & $8$ & $36$ & $324$ & $6$ & $101$ & $$ & $$ \\
   $Sym(m)$ & $5$ & $8422$ & $4$ & $5$ &$21$ & $133$ & $304$ & $6$ & $25$ & $153$ & $5$ & $69$ & $$ & $$  \\
    $\mathcal{I}_2(m)$ & $6$ & $22711$ & $4$ & $6$ &$29$ & $213$ & $990$ & $8$ & $36$ & $324$ & $6$ & $101$ & $$ & $$ \\
       \hline
\end{tabular}
\hspace{1cm}
\caption{Computations with skew braces of order at most 135}
\end{small}
\end{table}

%%%%%%%%%%%%%%%%%%%%%%%%%%%%%%%%%%%%%%%%%%%%%%%

\section{Symmetric skew braces}

In this section we first prove that all symmetric skew braces belong to the class $\tilde{\mathcal{I}}_n$ for all $n \ge 1$. Then we prove several results analogous to the results of Bioch \cite{Bioch78}.   We start with the following nice observation on symmetric skew braces, which follows from \cite[Proposition 5.2]{BNY22}

\begin{lemma}\label{lem:symm}
A skew brace $A$ is symmetric if and only if for all $a, b \in A$, we have $\lambda_{ab} = \lambda_{b \circ a}$. Moreover, in a symmetric skew brace $A$, (i)  $\lambda_{a^{-1}} = \lambda_{\bar{a}} = \lambda^{-1}_a$ and (ii) $\lambda_{a * b} = \lambda_{[a^{-1}, b]^{\cdot}}$ for all $a, b \in A$.
\end{lemma}

 It is not difficult to see that the left ideal $A_{(t)}$ constructed in the preceding section coincides with $\Gamma_{t+1}(A)$, when $A$ is a symmetric skew  brace. Indeed
 
\begin{lemma}
For a symmetric skew brace $A$, $A_{(t)}=\Gamma_{t+1}(A)$ for all  integers $t\ge 1$.
\end{lemma}
\begin{proof}
We prove it by induction. The result is obvious for $t =1$. So assume that $A_{(t-1)} = \Gamma_t(A)$. Also, it is sufficient to prove that the generators of $\Gamma_{t+1}(A)$ lie in $A_{(t)}$. If we take a generator $v$ that either belongs to $A * \Gamma_t(A)$ or $[A, \Gamma_t(A)]^{\cdot}$, then by the arguments used in the proof of Lemma \ref{lem:inclusion:1} we see that $v \in A_{(t)}$.  So assume that $v \in \Gamma_t(A) * A$.
So $v = (x_1 * a_1)^{\delta_1}  \cdots (x_r * a_r)^{\delta_r}$ for some $y_i \in \Gamma_t(A)$, $a_i \in A$ and positive integer $r$. So it is sufficient to prove it for any of the components of $v$. We prove it for the first component, that is, $(x_1 * a_1)^{\delta_1}$.

By the inductive argument, we know that $x _1 \in A_{(t-1)}$.  Let  $y,~y_1,~y_2 \in S_{t-1}(A)$ and $a \in A$. Then, using Lemma \ref{lem:symm}, we get
\begin{align*}
(y_1y_2)\ast a &=\lambda_{y_1y_2}(a)a^{-1} = \lambda_{y_2 \circ y_1} (a)a^{-1}\\
&=\lambda_{y_2}\big(\lambda_{y_1}(a)a^{-1}a\big)a^{-1}\\
&=(y_2\ast a)a\lambda_{y_2}(a(y_1\ast a)a^{-1})a^{-1},
\end{align*}
which belongs to $A_{(t)}$ as it is a strong left ideal of $A$. Now, using Lemma \ref{lem:symm} again, the preceding identity with $y_1 = y$ and $y_2 = y^{-1}$ gives
$$y^{-1}\ast a = a \lambda_{\bar{y}}\big(a(y * a) a^{-1})\big)^{-1} a^{-1},$$
which again lies in  $A_{(t)}$. Since $x_1$ can be written as a product of finitely many elements from $S_{t-1}(A) \cup (S_{t-1}(A))^{-1}$, this completes the proof.
\end{proof}

\begin{lemma}\label{lemma5}
Let $A$ be a symmetric skew brace. Then, for all $s \ge r \ge 1$,  $v \in \Gamma_r(A)$ and $u \in \Ann_s(A)$,  we have $v*u \in  \Ann_{s-r}(A)$.
\end{lemma}

\begin{proof}
We shall proceed by induction on $r$. For $r=1$, the result follows by the definition of the $s$th annihilator.  So assume that  $s \ge r \ge 2$ and the assertion holds for $r-1$.  We need to show that $\lambda_{v}(u)u^{-1} \in  \Ann_{s-r}(A)$ for all $v \in \Gamma_{r}(A)$ and $u \in \Ann_s(A)$. In view of Lemma \ref{lem:symm}, it is sufficient to prove the preceding assertion  for all generators of $\Gamma_{r}(A)$. Let  $v = \epsilon_1(a, d)$ for some $a \in A$ and $d \in  \Gamma_{r-1}(A)$, where $\epsilon_1 \in X$. If $\epsilon_1 = *$, then
\begin{eqnarray*}
\lambda_{a * d} (u) &=& \lambda_{a^{-1}(a \circ d) d^{-1}} (u) \\
&=&  \lambda_{d^{-1} \circ a \circ d \circ a^{-1}}(u)\\
&=&  \lambda_{d^{-1} \circ a \circ d}(wu), \mbox{ where $w \in \Ann_{s-1}(A)$ } \\
&=& \lambda_{d^{-1} \circ a} \big(yw \; x u \big), \mbox{ where $y \in \Ann_{s-r}(A), \;\;  x \in \Ann_{s-r+1}(A)$ } \\
&=& \lambda_{d^{-1}} \big( \lambda_{a} (ywxw^{-1})  \lambda_a(wu)\big) \\
&=& \lambda_{d^{-1}} \big(\lambda_a(y[w,x]^{\cdot} x) x^{-1} (xu)\big), \mbox{ since $\lambda_a(wu) = u$ } \\
&=& \lambda_{d^{-1}} \big(\lambda_a(y[w,x]^{\cdot}) \lambda_a(x) x^{-1} (xu)\big)\\
&=& \lambda_{d^{-1}} \big(\lambda_a(yz_1)z_2\big) \lambda_{d^{-1}}(xu), \mbox{ where $z_1, z_2 \in \Ann_{s-r}(A)$ }\\
&=& \lambda_{d^{-1}} \big(\lambda_a(yz_1)z_2\big) u, \mbox{ since $\lambda_{d^{-1}}(xu) = u$ } \\
&=& bu, \mbox{ where } b =  \lambda_{d^{-1}} \big(\lambda_a(yz_1)z_2\big) \in \Ann_{s-r}(A).
\end{eqnarray*}

Note that the preceding identity holds for any $a \in A$ and $d \in \Gamma_{r-1}(A)$. Thus there exists $b_1 \in \Ann_{s-r}(A)$ such that 
$$\lambda_{a^{-1} * d^{-1}} (u) = b_1 u.$$
Thus, by Lemma \ref{lem:symm},  we get
$$\lambda^{-1}_{[a, d^{-1}]^{\cdot}} (u) = \lambda^{-1}_{[a, d^{-1}]^{\cdot}}( b_1^{-1}) u.$$
Now, if we take  $\epsilon_1 = \bar{*}$, then, again using Lemma \ref{lem:symm}, we get
$$\lambda_{a \bar{*} d} (u) =  \lambda_{d * a} (u)    = \lambda_{[d^{-1}, a]^{\cdot}}(u) = \lambda^{-1}_{[a, d^{-1}]^{\cdot}} (u)   = \lambda^{-1}_{[a, d^{-1}]^{\cdot}}( b_1^{-1}) u,$$
where $\lambda^{-1}_{[a, d^{-1}]^{\cdot}}( b_1^{-1})\in \Ann_{s-r}(A)$. Finally, the cases  $\epsilon_1 = \gamma_., \bar{\gamma}_.$ follow from the  identity $\lambda_{a * d} = \lambda_{[a^{-1}, d]^{\cdot}}$. The proof is now complete. 
\end{proof}

\begin{thm}
Let  $A$ be a symmetric skew brace. Then $A \in \tilde{\mathcal{I}}_n$ for all $n \ge 1$. 
\end{thm}
\begin{proof}
In view of Theorem \ref{phi-well-defined}, to prove that $A \in \bar{\mathcal{I}}_n$, it is sufficient to prove that $\epsilon(u,v) \in \Ann_{s-r}(A)$, for all $\epsilon \in X$,  integers $s \ge r \ge 1$,  $u \in \Ann_s(A)$ and $v \in \Gamma_r(A)$.  
We prove this result using induction on $r$. The case $r=1$ holds true by the definition of the annihilator series. Assume that the assertion holds for all values $\le r-1$ given any $s \ge r-1$.   Let $u \in  \Ann_s(A)$ and $v  \in \Gamma_{r}(A)$, $s \ge r \ge 1$. Write $v = \epsilon_1(a, d)$, where $d \in \Gamma_{r-1}(A)$.  By our assumption $\epsilon(u, d)  \in \Ann_{s-r+1}(A)$.  
The case $\epsilon = \bar{*}$ has already been handled in Lemma \ref{lemma5}. The cases $\epsilon = *$ and $\epsilon_1 = \gamma_{\cdot}, *$ have been taken care of in Lemma \ref{lemma4}. There is word of caution on the use of Lemma \ref{lemma4}. Although the domain of $v$ is bigger in the present situation, the computations are precisely the same. The cases $\epsilon = *$ and $\epsilon_1 = \bar{\gamma}_{\cdot}, \bar{*}$ will follow along similar lines, and are therefore left for the reader as an easy exercise.

The cases $\epsilon = \gamma_{\cdot}$ and $\epsilon_1 = \gamma_{\cdot}, *$ have also been handled in Lemma \ref{lemma4},  and the cases  $\epsilon = \gamma_{\cdot}$ and $\epsilon_1 = \bar{\gamma}_{\cdot}$ will follow similarly. As above, some caution is required on the use of Lemma \ref{lemma4}.

Finally we take $\epsilon = \gamma_{\cdot}$, $\epsilon_1 =  \bar{*}$. Notice that
\begin{eqnarray}\label{eqn1}
 [\lambda_{d}(a), [u, a^{-1}]^{\cdot}]^{\cdot} &=&  \lambda_{d}\big([a, \lambda^{-1}_d([u, a^{-1}]^{\cdot})]^{\cdot}\big) \nonumber\\
 &=&  \lambda_{d}\big(\big[a, \lambda^{-1}_d([u, a^{-1}]^{\cdot})([u, a^{-1}]^{\cdot})^{-1}[u, a^{-1}]^{\cdot}\big]^{\cdot}\big) \nonumber\\
 &=&  \lambda_{d}\big([a, z_1[u, a^{-1}]^{\cdot}]^{\cdot}\big),\mbox{ where $z_1 \in \Ann_{s-r}(A)$ } \nonumber\\
 &=& \lambda_{d}\big([a, z_1]^{\cdot} [z_1, [a, [u, a^{-1}]^{\cdot}]^{\cdot}]^{\cdot} [a, [u, a^{-1}]^{\cdot}]^{\cdot}\big) \nonumber\\
 &=& z_2 \big(\lambda_d([a, [u, a^{-1}]^{\cdot}]^{\cdot})([a, [u, a^{-1}]^{\cdot}]^{\cdot})^{-1}\big) [a, [u, a^{-1}]^{\cdot}]^{\cdot} \nonumber\\
 &=& z_2 z_3 [a, [u, a^{-1}]^{\cdot}]^{\cdot},
 \end{eqnarray} 
where $z_2 z_3 \in \Ann_{s-r}(A)$. Now

\begin{eqnarray*}
[u, a \bar{*} d]^{\cdot} &=& [u, \lambda_{d}(a)a^{-1}]^{\cdot} \\
&=& [u, \lambda_{d}(a)]^{\cdot} [\lambda_{d}(a), [u, a^{-1}]^{\cdot}]^{\cdot} [u, a^{-1}]^{\cdot}\\
&=& \lambda_{d}([\lambda^{-1}_{d}(u), a]^{\cdot}) [\lambda_{d}(a), [u, a^{-1}]^{\cdot}]^{\cdot}([a, [u,a^{-1}]^{\cdot}]^{\cdot})^{-1} ([u, a]^{\cdot})^{-1}\\
&=& \lambda_d([wu, a]^{\cdot}) z_2 z_3 ([u, a]^{\cdot})^{-1}, \mbox{ (by \eqref{eqn1}) where $w \in \Ann_{s-r+1}(A)$ } \\
&=& \lambda_d\big([w, [u, a]^{\cdot}]^{\cdot} [u, a]^{\cdot} [w, a]^{\cdot}\big) z_2z_3 ([u, a]^{\cdot})^{-1}\\
&=&  \lambda_d([w, [u, a]^{\cdot}]^{\cdot}) \lambda_d([u, a]^{\cdot}) \lambda_d ([w, a]^{\cdot}) z_2z_3 ([u, a]^{\cdot})^{-1}\\
&=&  \lambda_d([w, [u, a]^{\cdot}]^{\cdot})  (\lambda_d ([w, a]^{\cdot}) z_2z_3 )^{\lambda_d([u, a]^{\cdot})} \lambda_d([u, a]^{\cdot})([u, a]^{\cdot})^{-1}\\
&=&  \lambda_d([w, [u, a]^{\cdot}]^{\cdot})  \big(\lambda_d ([w, a]^{\cdot}) z_2z_3 \big)^{\lambda_d([u, a]^{\cdot})} (d * [u, a]^{\cdot}),
\end{eqnarray*}
which lies in $\Ann_{s-r}(A)$, since $[w, [u, a]^{\cdot}]^{\cdot}$, $[w, a]^{\cdot}$ and $d * [u, a]^{\cdot}$  all lie in  $\Ann_{s-r}(A)$ by inductive argument. The case $\epsilon = \bar{\gamma}_.$ follows from the case $\epsilon = \gamma_.$ This proves that  $A \in \bar{\mathcal{I}}_n$.  That the class of symmetric skew braces lies in $\tilde{\mathcal{I}}_n$ follows by using  Lemma \ref{lem:symm}, and the proof is complete.
 \end{proof}

 So, we can now talk about $n$-isoclinism of symmetric skew braces. Numerous results have been proved on $n$-isoclinism of groups in the literature. We prove here some results analogous to those results for skew braces. The following three results are analogous to \cite[Lemma 1.2, Lemma 1.3 and Theorem 1.4]{Bioch78}. In the case of symmetric skew braces $A$, we know that $A_{(n)} = \Gamma_{n+1}(A)$. We'll mostly use  $\Gamma_{n+1}(A)$  for  $A_{(n)}$ while dealing with $n$-isoclinism.

\begin{prop}\label{prop6}
Let $A$ and $B$ be $n$-isoclinic symmetric skew  braces with $n$-isoclinism  $(\xi, \theta)$. Then the following statements hold true:
\begin{enumerate}
\item If $A_1$ is sub-skew  brace of $A$ containing $\Ann_n(A)$ and $\xi(A_1/\Ann_n(A)) = B_1/\Ann_n(B)$, then  $B_1$ is a  sub-skew   brace of $B$ containing $\Ann_n(B)$ and $A_1$ and $B_1$ are $n$-isoclinic skew braces. 
\item  If $A_2$ is an ideal of $A$ such that $A_2 \le \Gamma_{n+1}(A)$, then $\theta(A_2) \le  \Gamma_{n+1}(B)$ is an ideal of $B$ and  the skew  braces $A/\A_2$ and $B/\theta(A_2)$ are $n$-isoclinic.
\end{enumerate}
\end{prop}

We first prove a lemma.
\begin{lemma}\label{lemma6}
With the settings of Proposition \ref{prop6}, we have 
$$\xi\big(\Ann_n(A) m(a_1 \cdots a_{n+1})\big) =  \Ann_n(B) \theta \big(m(a_1 \cdots a_{n+1})\big)$$
 for all $a_1, \ldots a_{n+1} \in A$ and $m \in M$ with $\deg(m) = n$.
\end{lemma}
\begin{proof}
With the given hypotheses, we get
\begin{eqnarray*}
\xi\big(\Ann_n(A) m(a_1 \cdots a_{n+1})\big) &=&  m\big(\xi(\tilde{a}_1) \cdots \xi(\tilde{a}_{n+1})\big)\\
&=& m(\tilde{b}_1 \cdots \tilde{b}_{n+1})\\
&=& \Ann_n(B) m(b_1 \cdots b_{n+1})\\
&=& \Ann_n(B) \phi_m^B(\tilde{b}_1,  \ldots, \tilde{b}_{n+1})\\
  &=&   \Ann_n(B) \theta \big(m(a_1 \cdots a_{n+1})\big),
  \end{eqnarray*}
 where $\tilde{a}_i = \Ann_n(A)a_i $ and $\tilde{b}_i = \xi(\tilde{a}_i)$, $1 \le i \le n+1$. 
\end{proof}

\noindent {\it Proof of Proposition \ref{prop6}.}
That $B_1$ is a symmetric sub-skew  brace, is easy to see. Notice that $\Ann_n(A) \le \Ann_n(A_1)$ and $\Ann_n(B) \le \Ann_n(B_1)$. Also $\Gamma_{n+1}(A_1) \le \Gamma_{n+1}(A)$ and $\Gamma_{n+1}(B_1) \le \Gamma_{n+1}(B)$. Define 
\begin{eqnarray*}
\bar{\xi} &:& A_1/\Ann_n(A_1) \to B_1/\Ann_n(B_1)\\
\bar{\theta} &:& \Gamma_{n+1}(A_1) \to \Gamma_{n+1}(B_1)
\end{eqnarray*}
 by setting, for any $a_1 \in A_1$ and $v_1 \in \Gamma_{n+1}(A_1)$,
\begin{eqnarray*}
\bar{\xi}(\Ann_n(A_1)a_1) &=&  \Ann_n(B_1)a_2, \mbox{ where $a_2$ is a coset representative of $\xi(\Ann_n(A)a_1)$ },\\
\bar{\theta}(v_1) &=& v_2, \mbox{ where $v_2 = \theta(v_1) $}.
\end{eqnarray*}

Let $a_1 \not\in \Ann_n(A_1)$ such that $a_2 \in \Ann_n(B_1) - \Ann_n(B)$.  Let  $x_1, x_2, \ldots x_n \in A_1$ be arbitrary elements. Then 
by the given hypothesis, there exist $y_1, y_2, \ldots y_n \in B_1$ such that $\xi(\Ann_n(A) x_i) = \Ann_n(B) y_i$. Since  $a_2 \in \Ann_n(B_1)$, we have  $m(y_1 y_2 \cdots y_n a_2) = 1$ for all $m \in M$ with $\deg(m) = n$.  Further, by the definition of $n$-isoclinism $m(x_1 x_2 \cdots x_n a_2) = 1$. Since $\xi$ is a brace isomorphism, it follows that $a_1 \in \Ann_n(A_1)$. This is a contradiction, which proves that $\bar{\xi}$ is well defined and one-one. That  $\bar{\xi}$ is onto can be easily established by the similar argument.  It is easy, using the definition of $n$-isoclinism of $A$ and $B$, that  $\bar{\theta}$ is an isomorphism, and $A_1$ and $B_1$ are $n$-isoclinic.
 
Since $\theta$ is a brace isomorphism, $\theta(A_2)$ is a subgroup of  both $(B, \cdot)$ and $(B, \circ)$. To show that $\theta(A_2)$ is an ideal of $B$, it is sufficient to check $v_2 * b$ and $b* v_2$ belong to $\theta(A_2)$ for all $b \in B$ and $v_2 \in \theta(A_2)$. Let $b \in B$ be an arbitrary element,  $a \in A$ such that $\xi(\Ann_n(A)a) = \Ann_n(B)b$ and $v_1 \in \Gamma_{n+1}(A)$ such that $\theta(v_1) = v_2$.  Then it follows that $\theta(a * v_1) = b * v_2$ and $\theta(v_1 * a) = v_2 * b$. We remark that this conclusion is not as easy as it seems, because $\theta$ is  not an isomorphism of  $A$ onto $B$. Tricky computations and the definition of $n$-isoclinism are involved. Since, $A_2$ being an ideal of $A$, both $a * v_1$ and $v_1 * a$ lie in $A_2$, we have that $v_2 * b$ and $b* v_2$ belong to $\theta(A_2)$.

Set $\tilde{A} := A/A_2$ and $\tilde{B} := B/\theta(A_2)$. We now define
\begin{eqnarray*}
\tilde{\xi} &:& \tilde{A}/ \Ann_n(\tilde{A}) \to  \tilde{B}/ \Ann_n(\tilde{B})\\
\tilde{\theta} &:& \Gamma_n(\tilde{A}) \to \Gamma_n(\tilde{B})
\end{eqnarray*}
by
\begin{eqnarray*}
\tilde{\xi}(\Ann_n(\tilde{A})\tilde{a}_1) &=&  \Ann_n(\tilde{B})\tilde{a}_2, \mbox{ where $a_2$ is a coset representative of $\xi(\Ann_n(A)a_1)$},\\
\tilde{\theta}(\tilde{v}_1) &=& \tilde{v}_2, \mbox{ where $v_2 = \theta(v_1)$}.
\end{eqnarray*}
That  $\tilde{\theta}$ is well-defined and an isomorphism is easy to see. We claim  that the restriction of  $\xi$ to $(A_2 \Ann_n(A))/$ $\Ann_n(A)$ is determined by the map induced by $\theta$ on $(A_2\Ann_n(A))/\Ann_n(A)$ (note that  $A_2 \le \Gamma_{n+1}(A)$). It is sufficient to check it on the values of the words $m \in M$ with $\deg(m) = n$, which actually follows from Lemma \ref{lemma6}. This settles our claim.
 Now considering the fact that $n$-isoclinism is a symmetric relation, we see that  $\bar{\xi}$ is well defined as well as an isomorphism. Commutativity of the diagrams now follows from the definition of $n$-isoclinism of $A$ and $B$.   \hfill $\Box$

\begin{prop}\label{prop7}
Let $A$ be a  symmetric skew brace with a sub-skew brace $B$ and let $I$ be an ideal of $A$.  Then 
\begin{enumerate}
\item  $B$ and $B\Ann_n(A)$ are $n$-isoclinic. Consequently, if $A =  B\Ann_n(A)$, then $A$ and $B$ are $n$-isoclinic. Conversely, if $A/\Ann_n(A)$ is finite and $A$ is isoclinic to $B$, then $A = B\Ann_n(A)$.
\item  $A/I$ is $n$-isoclinic to $A/(I \cap \Gamma_{n+1}(A))$.  Consequently, if $I \cap \Gamma_{n+1}(A) =1$, then $A$ and $A/I$ are $n$-isoclinic. Conversely, if $\Gamma_{n+1}(A)$ is finite and $A$ and $A/I$ are $n$-isoclinic, then $I \cap \Gamma_{n+1}(A) =1$.
\end{enumerate}
\end{prop}

\begin{proof}
By Lemma  \ref{lemma1} we know that $B\Ann_n(A)$ is a sub-skew  brace of $A$. We start by noting that $\Ann_n(B\Ann_n(A)) = \Ann_n(B)\Ann_n(A)$. Define, for all $b \in B$, 
$$\xi : B/\Ann_n(B) \to  B\Ann_n(A)/\Ann_n(B\Ann_n(A)) \;\; \big(= B\Ann_n(A)/(\Ann_n(B)\Ann_n(A))\big)$$
  by
$$\xi\big(\Ann_n(B)b\big) = \Ann_n(B\Ann_n(A))b = (\Ann_n(B)\Ann_n(A))b,$$
which is naturally an isomorphism under the operation `$\cdot$'. To show that $\xi$ is a brace homomorphism, we only need to check that $\xi$ is a homomorphism under `$\circ$'.
For $b_1, b_2 \in B$, we have
\begin{eqnarray*}
\xi\big(\Ann(B) (b_1 \circ b_2) &=&  (\Ann_n(B)\Ann_n(A)) (b_1 \lambda_{b_1}(b_2))\\
&=& \big(\Ann_n(B)\Ann_n(A) b_1\big) \big(\Ann_n(B)\Ann_n(A)) \lambda_{b_1}(b_2)\big)\\
&=&  \big(\Ann_n(B)\Ann_n(A) b_1\big) \circ \lambda^{-1}_{\tilde{b}_1}\big(\lambda_{\tilde{b}_1}(\Ann_n(B)\Ann_n(A)) b_2\big)\\
&=& \xi(b_1) \circ \xi(b_2),
\end{eqnarray*}
where $\tilde{b}_1 := \Ann_n(B)\Ann_n(A) b_1$. Thus $\xi$ is a brace isomorphism. Taking 
$$\theta : \Gamma_{n+1}(B) \to \Gamma_{n+1}(B\Ann_n(A)) (= \Gamma_{n+1}(B))$$
 to be the identity map, it follows that  $B$ and $B\Ann_n(A)$ are $n$-isoclinic.

Conversely, let $A$ and $B$ be $n$-isoclinic. If $\Ann_n(A) \not\subseteq B$, then we replace $B$ by $B\Ann_n(A)$ to assume that  $\Ann_n(A) \subseteq B$.  By the preceding paragraph,  we see that $A$ and $B$ are still $n$-isoclinic. So we only need to prove that $A = B$.
Contrarily assume that $B  \neq A$. Note that $\Ann_n(A) \le \Ann_n(B) \le B \le A$. Let $|A/\Ann_n(A)| = |B/\Ann_n(B)| = n$. Then
\begin{eqnarray*}
n &=& |B/\Ann_n(B)| = |B/\Ann_n(A)|/|\Ann_n(B)/\Ann_n(A)| \le |B/\Ann_n(A)| \\
&<& |A/\Ann_n(A)| = n,
\end{eqnarray*}
a contradiction. Hence $A = B$.

Let us set $\hat{A} := A/I$ and $\tilde{A} := A/(I \cap \Gamma_{n+1}(A))$. Note that, two elements $v_1, v_2 \in A_{(n)}$ are equal in $\hat{A}$  if and only if $v_1, v_2$ are equal in  
$\tilde{A}$. Also   any $\hat{u} \in \hat{A}$ belongs to $\Ann_n(\hat{A})$ if and only if $v:=m(a_1 \cdots a_n u) \in I$ for all choices of $a_1, \ldots, a_n \in A$ and  $m \in M$ with $\deg(m) = n$. Since $v \in \Gamma_{n+1}(A)$,  we obviously have that $\hat{v} =1$ if and only if $\tilde{v} = 1$.  We have shown that $\hat{u} \in \Ann_n(\hat{A})$ if and only if  $\tilde{u} \in \Ann_n(\tilde{A})$. Hence the map $\xi : \hat{A}/\Ann_n(\hat{A}) \to \tilde{A}/ \Ann_n(\tilde{A})$, defined by
$$\xi(\Ann_n(\hat{A})\hat{a}) = \Ann_n(\tilde{A})\tilde{a}$$
for all $a \in A$, is an isomorphism. Also the map $\theta : \Gamma_{n+1}(\hat{A}) \to \Gamma_{n+1}(\tilde{A})$, defined by
$$\theta(\hat{v}) = \tilde{v}$$
for all $\hat{v} \in \Gamma_{n+1}(\hat{A})$, is an isomorphism. It is now easy to see that the pair $(\xi, \theta)$ gives rise to an $n$-isoclinism of $\hat{A}$ and $\tilde{A}$. The converse part is an easy consequence of the expression
$$\Gamma_{n+1}(A) \cong \Gamma_{n+1}(A/I) \cong \Gamma_{n+1}(A)I/I \cong \Gamma_{n+1}(A)/(I \cap \Gamma_{n+1}(A)).$$
This completes the proof.   
\end{proof}

Please note that the symbol $\hat{A}$, in the preceding proof, for $A/I$ is used locally only for this proof. Anywhere else it means $\Hom(A, \mathbb{C}^*)$.

We are now ready to prove Theorem B.
\begin{thm}\label{thm:isoclin-surj}
Let $A, B$ be symmetric  skew braces. Then $A$ and $B$ are $n$-isoclinic if and only if there exist  skew braces $C$, $N_1$, $N_2$, $C_A$ and $C_B$ such that $A \cong C/N_2$ and $B \cong C/N_1$ and the following two equivalent conditions hold:
\begin{enumerate}
\item  $C$ is $n$-isoclinic to both $A$ and $B$.
\item $C_B \cong C \cong C_A$, and $C$ is $n$-isoclinic to both $C/N_2 \times C/\Gamma_{n+1}(C)$ and  $C/N_1 \times C/\Gamma_{n+1}(C)$, where $C_B$ and $C_A$ are sub-skew  braces of $C/N_2 \times C/\Gamma_{n+1}(C)$ and  $C/N_1 \times C/\Gamma_{n+1}(C)$ respectively.
\end{enumerate}
\end{thm}
\begin{proof}
We only need to prove the `only if' part of the statement. So let  us assume that $A$ and $B$ are $n$-isoclinic with  $n$-isoclinism  $(\xi, \theta)$.   Define 
$$C:= \{(a, b) \in A \times B \mid \xi(\Ann_n(A)a)= \Ann_n(B)b\},$$
 the fiber product of $A$ and $B$   over $A/\Ann_n(A)$ (which we assume equal to $B/\Ann_n(B)$ through $\xi$). It turns out that $C$ is a (symmetric) skew left brace.  Note that $\Ann_n(A) \times \Ann_n(B) \subseteq C$.  Set  $N_1 := \{(u, 1) \mid u \in \Ann_n(A)\} \cong \Ann_n(A)$ and $N_2 := \{(1, u) \mid u \in \Ann_n(B)\} \cong \Ann_n(B)$.  Obviously, both $N_1$ and $N_2$ are ideals of $C$. So if we factor $C$ by $N_2$, then for every $a \in A$ there will be a unique $b \in B$ such that $\tilde{(a, b)} \in C/N_2$, where $\tilde{(a, b)}$ is the right coset of $N_2$ by the element $(a, b)$ of $C$. Hence $C/N_2 \cong A$. Similarly, $C/N_1 \cong B$.

To prove assertion (1) we only need to prove that $C$ is $n$-isoclinic to both $C/N_1$ and $C/N_2$. By Proposition \ref{prop7}(2) it suffice to prove that 
$$N_1 \cap \; \Gamma_{n+1}(C) = 1 =  N_2 \cap \; \Gamma_{n+1}(C).$$
 Let $m(a_1 \cdots a_{n+1}) \in \Gamma_{n+1}(A)$. Then, we know by Lemma \ref{lemma6} that 
$$\xi\big(\Ann_n(A)m(a_1 \cdots a_{n+1})\big) = \Ann_n(B) \theta\big(m(a_1 \cdots a_{n+1})\big).$$
Thus it follows that $\Gamma_{n+1}(C)$ is generated by the elements of the form 
$$\big(m(a_1 \cdots a_{n+1}), \theta(m(a_1 \cdots a_{n+1}))\big).$$
Now suppose that $(u, 1) \in N_1  \cap \; \Gamma_{n+1}(C)$. Then there exists $v \in \Gamma_{n+1}(A)$ such that $(u, 1) = (v, \theta(v))$, which implies that $u = v = 1$. Hence $N_1 \cap \; \Gamma_{n+1}(C) = 1$. Similarly,  we have   $N_2 \cap \; \Gamma_{n+1}(C) = 1$.

Now consider the skew left brace $C/N_1 \times C/\Gamma_{n+1}(C)$. Let $C_A := \{N_1 c, \Gamma_{n+1}(C)c \mid c \in C\}$. Since $N_1 \cap \; \Gamma_{n+1}(C) = 1$, it follows that $C_A \cong C$ as a skew brace. Notice that $C_A \Ann_n\big(C/N_1 \times C/\Gamma_{n+1}(C)\big) = C/N_1 \times C/\Gamma_{n+1}(C)$, since $\Ann_n(C/\Gamma_{n+1}(C)) = C/\Gamma_{n+1}(C)$. By Proposition \ref{prop7}(1) we see that $C_A$ is $n$-isoclinic to $C/N_1 \times C/\Gamma_{n+1}(C)$. The other part of assertion (2) follows on the same lines. This proves the theorem. 
\end{proof}

The preceding result guarantees the existence of a skew brace $C$ for given symmetric skew braces $A$ and $B$ such that both $A$ and $B$ are brace homomorphic images of $C$. An inclusion result is also possible, that is, for two symmetric skew braces $A$ and $B$, there exists a skew brace $C$ such that both $A$ and $B$ embed in $C$. We take this up in the next section.

%%%%%%%%%%%%%%%%%%%%%%%%%%%%%%%%%%%%%%%%%%%%%%%%%%%%%%%%%%%%%

\section{Embedding theorem for isoclinic skew  braces}

Catching the vibes from the preceding section, in this section we investigate  embedding theorems for $n$-isoclinic skew braces. For $n=1$, we could prove the result analogous to group theory results from \cite{Tappe76} in full generality.  But for $n \ge 2$, we could only do so for symmetric skew braces, which are analogous to the results of \cite{Waall2004}.  First we take the case $n = 1$ and prove Theorem C.

\begin{thm}\label{thm:embed}
Let $A$ and $B$ be isoclinic skew  braces. Then there exists a skew  brace $W$ with sub-skew  braces $\bar{A}$ and $\bar{B}$ such that $A\cong \bar{A},~B\cong \bar{B},$ and $\bar{A}  \Ann(W)=\bar{A}\circ \Ann(W) = W = \bar{B} \circ \Ann(W) = \bar{B}  \Ann(W)$.
\end{thm}
\begin{proof}
Let $(\xi, \theta)$ be a skew brace isoclinism. Let 
$$C:=\{(a,b)\in A\times B \mid \xi(\Ann(A)a )= \Ann(B) b ,~a\in A,~b\in B\}.$$
It is not very difficult to prove that $C$ is a sub-skew  brace of $A \times B$.  Consider the maps $\pi_A:C\to A,~\pi_B:C\to B$ defined by $\pi_A(a,b) = a$ and $\pi_B(a,b) = b$ respectively. It is clear that $\pi_A,\pi_B$ are brace homomorphisms;  since $\xi$ is an isomorphism, these are surjective brace homomorphisms. Set  $N_2:= \Ker(\pi_A)$ and $N_1:= \Ker(\pi_B)$. It follows that
$$N_1=\{(a, 1)\in C \mid a\in \Ann A\},~N_2=\{(1, b)\in C \mid b\in \Ann B\}.$$
The structure of  $\Gamma_2(C)$ can be simplified as follows.
\begin{align*}
\Gamma_2(C) &=\langle [C,C]^{\cdot},C\ast C\rangle^{\cdot}\\
&=\langle [(a_1,b_1),(a_2,b_2)]^{\cdot}, (a_1,b_1)\ast (a_2,b_2) \mid (a_i,b_i)\in C\rangle^{\cdot}\\
&=\langle ([a_1,a_2]^{\cdot},[b_1,b_2]^{\cdot}), (a_1\ast a_2,b_1\ast b_2) \mid (a_i,b_i)\in C\rangle^{\cdot}\\
&=\{(a,\theta(a)) \mid a\in \Gamma_2(A)\}.
\end{align*} 
Set  $Y := C/{N_2} \times C/{\Gamma_2(C)}$ and $$N := \langle \big(  N_2(a,1),  \Gamma_2(C)(a, 1)\big) \mid (a,1) \in N_1\rangle^{\cdot}.$$
We claim that $N$ is an ideal of $Y$.  Since $N_1$ is an ideal of $C$, we have $c t c^{-1},~c \circ t \circ \bar{c} \in N_1$, on the other hand $\Gamma_2(C) c t c^{-1} = \Gamma_2(C)t=\Gamma_2(C) (c \circ t \circ \bar{c})$  for all $c \in C$ and $t \in N_1$, it follows that $N$ is a normal subgroup of $Y$ with respect to `$\cdot$' as well as `$\circ$'. A routine check gives invariance under the $\lambda$ map.

Let $W := Y/N$. We now prove that the skew left braces $A$ and $B$ embed in  $W$.  For any $(a, b_1), (a, b_2) \in C$, $(a, b_1) (a, b_2)^{-1} = (1, b_1 b_2^{-1}) \in N_2$. Thus the map $\rho_A : A \to W$ given by 
$$\rho_A(a) = N\big( N_2(a, b),  \Gamma_2(C)(1,1)\big) $$
is well defined.  Now let $(a_1,b), (a_2,b) \in C$. Then  $(a_1, b) (a_2, b)^{-1} = (a_1 a_2^{-1}, 1) \in N_1$.  Thus we get a well defined  map  $\rho_B : B \to W$ given by 
$$\rho_B(b) = N\big( N_2(a,b),   \Gamma_2(C)(a,b)\big).$$ 
An easy check now establishes that $\rho_A$ and $\rho_B$ are brace homomorphisms.

We only prove that $\rho_A$ is an embedding. Verification of the other case is left for the reader.  Let $\rho_A(a) = 1$ for some $a\in A$. Then $\big( N_2(a,b),   \Gamma_2(C)(1,1)\big) \in N$; so there exists an  element  $c \in \Ann(A)$ such that $  N_2(a,b) =   N_2(c,1)$ and $ \Gamma_2(C) (1,1)= \Gamma_2(C)(c,1)$. Hence $(c,1) \in \Gamma_2(C)$. Since $\theta$ is an isomorphism, we have   $c=1$. This yields $(a,b) \in N_2$, that is, $a=1$. Hence $\rho_A$ is an embedding.

It only remains to prove the last assertion. We only work it out for  $(A, \cdot)$, that is,   $\rho_A(A) \Ann(W)  = W$.
Let $u =  N \big( N_2(a_1, b_1),  \Gamma_2(C)(a_2, b_2)\big)   \in W$ be an arbitrary element. Then, obviously, we can write
$$u =  \big( N \big( N_2(a_1, b_1),  \Gamma_2(C)(1, 1)\big)\big(N \big( N_2(1, 1),  \Gamma_2(C)(a_2, b_2)\big) \big)  \big).$$
The first part in the product is obviously  $\rho_A(a_1)$. A straightforward computation, considering the fact that operations take place modulo $\Gamma_2(C)$, reveals that $v := N\big( N_2(1, 1),  \Gamma_2(C)(a_2, b_2)\big)  \in \Z(W, \cdot) \cap \Z(W, \circ)$. Since $C/\Gamma_2(C)$ is a trivial skew brace, we can easily see that $w\circ v = w  v$ for all $w \in W$. Hence $v \in \Ann(W)$, which proved our intended assertion. The proof is now complete by taking $\rho_A(A) = \bar{A}$ and $\rho_B(B) = \bar{B}$.
\end{proof}

\begin{rmk}\label{rmk:embed1}
Let $K:=LN/N$, where $L:=\{\big(N_2,\Gamma_2(C)(c_1,c_2)\big)~|~(c_1,c_2)\in C\}$. Then $\Gamma_2(K) = 1$, $[\rho_A(A),K]^\cdot=1=[\rho_A(A),K]^\circ$, $W=\rho_A(A)K$ and both $\rho_A(A)$ and $K$ are ideals of $W$.
\end{rmk}

The following result reproves   \cite[Proposition 3.10]{LV23}. 
\begin{cor}
Let $A$ and $B$ be isoclinic skew braces. Then  $(A, \cdot)$ is isoclinic to  $(B, \cdot)$  and  $(A, \circ)$ is isoclinic to $(B, \circ)$. \end{cor}
\begin{proof}
Without loss of generality, considering the setup of  Theorem \ref{thm:embed}, we can assume $(A, \circ), (B, \circ)$ are subgroups of $(W, \circ)$ such that $(A, \circ) \circ \Z(W, \circ) = (W, \circ) = (B, \circ)  \circ \Z(W, \circ)$. Notice that $[A, A]^{\circ} = [W, W]^{\circ}$.  Consider the map $\alpha: (W, \circ)/{\Z(W, \circ)} \to (A, \circ)/{\Z(A,\circ)}$ given by 
$$\alpha( \Z(W, \circ)\circ (u \circ v)) =  \Z(A, \circ)\circ u$$
for all $u \in (A, \circ)$ and $v \in \Z(W, \circ)$. Since $\Z(A, \circ) \subseteq \Z(W, \circ)$, it follows that $\alpha$ is a group isomorphism.
An easy check now establishes that  $(A, \circ)$ is isoclinic to $(W, \circ)$. Similarly, one can show that $(B, \circ)$ is isoclinic to $(W, \circ)$. Since  isoclinism is an equivalence relation, we have that $(A, \circ)$ is isoclinic to $(B, \circ)$. Verbatim for the other assertion  completes the proof.
\end{proof}

From the following theorem we obtain that isoclinism of skew braces can also be defined using ideals inside annihilators instead of the full annihilators.

\begin{thm}\label{thm:Tappe}
Two skew braces $A$ and $B$ are isoclinic if and only if there exist ideals $L \subseteq \Ann A, ~N\subseteq \Ann B$ of $A,~B$ respectively and isomorphisms $\xi:A/L\to B/N$ and $\theta: \Gamma_2(A) \to \Gamma_2(B)$ such that  the diagram
\begin{equation}\label{dia:Isoclinism2}
\begin{tikzcd}
\Gamma_2(A)\arrow[d,"\theta"] &(A/{L})^2\arrow[l,"\phi_{\cdot}^A"']\arrow[r,"\phi_\ast^A"]\arrow[d,"\xi \times \xi"] & \Gamma_2(A)\arrow[d,"\theta"]\\
\Gamma_2(B) &(B/{N})^2\arrow[l,"\phi_{\cdot}^B"]\arrow[r,"\phi_\ast^B"'] &\Gamma_2(B)
\end{tikzcd}
\end{equation}
commutes.
\end{thm}
\begin{proof}
We only aim to prove the non-trivial implication. We are going to show that $\xi$ induces an isomorphism $\bar{\xi} : A/{\Ann(A)}\to B/{\Ann(B)}$ such that $(\bar{\xi}, \theta)$ is a brace isoclinism of $A$ and $B$. Let $u \in \Ann(A)$ and $a \in A$. Then 
$$[a, u]^{\cdot} =1 = a \ast u = u \ast a.$$
By the commutativity of  diagram \eqref{dia:Isoclinism2}, we have
$$[\tilde{a}, \tilde{u}]^{\cdot} = 1 = \tilde{a} \ast \tilde{u} =  \tilde{u} \ast \tilde{a},$$
where $\tilde{x}$ is the coset representative of $\xi(\bar{x})$ in $B$ for $x \in \{a, u\}$.
Since $a \in A$ is arbitrary and $\xi$ is an isomorphism, we get $\tilde{u} \in \Ann(B)$, which implies that $\xi(\Ann(A)/L) \subseteq \Ann(B)/N$.  The reverse inclusion holds by considering $\xi^{-1}$  and $\theta^{-1}$. Hence $\xi(\Ann(A)/L) = \Ann(B)/N$. The desired  isomorphism $\bar{\xi}$ is given by 
$$\bar{\xi}(\Ann(A) \cdot a) =   \Ann(B) \cdot \tilde{a};  \,\,\,\,\, \bar{\xi}(\Ann(A) \circ a) =  \Ann(B) \circ \tilde{a}$$
for all $a \in A$. It is easy to see that  $(\bar{\xi}, \theta)$ is a brace isoclinism of $A$ and $B$.
\end{proof}

\begin{lemma}\label{lem:PropertyIsocli}
Let $A$ and $B$ be skew braces with ideals $L \subseteq \Ann A$ and $N\subseteq \Ann B$. Assume there exist isomorphisms $\xi:A/L\to B/N$ and $\theta: \Gamma_2(A) \to \Gamma_2(B)$ such that diagram \eqref{dia:Isoclinism2} commutes. Then we have
\begin{enumerate}
\item $a\in \Gamma_2(A)$ implies $\xi(L  a)=  N  \theta(a),$
\item $\theta(L \cap \Gamma_2(A))=N\cap \Gamma_2(B)$.
\end{enumerate}
\end{lemma}
\begin{proof}
Let $a_1, a_2\in A$ and $\xi(L  a_1) = N  b_1,~\xi( L  a_2) = N   b_2$. Then the commutative diagram gives $\xi(L [a_1,a_2]^{\cdot} ) =  N [b_1,b_2]^{\cdot}  = N \theta([a_1,a_2]^{\cdot}) $ and $\xi( L  (a_1\ast a_2)) = N  (b_1 \ast b_2)  = N  \theta(a_1\ast a_2)$. This gives the  first assertion.

Let $a\in L \cap \Gamma_2(A)$, then from the first part we have $N  1= \xi(L  a) = N  \theta(a)$. Thus $\theta(a) \in N \cap \Gamma_2(B)$. Now assume $b \in N \cap \Gamma_2(B)$, then again from the first part we have $\xi(L \theta^{-1}(b)) = N  \theta\theta^{-1}(b) =  N  b = N 1$. Therefore $\theta^{-1}(b)\in L \cap \Gamma_2(A)$, that is, $b\in \theta(L \cap \Gamma_2(A))$, which completes the proof. \end{proof}

We now turn our attention to the case $n \ge 1$.  Theorem   \ref{thm:Tappe} can be generalized   for $n$-isoclinic skew braces as follows:

\begin{thm}\label{thm:Hekster}
Two symmetric skew braces $A, B$ are $n$-isoclinic if and only if there exist ideals $L \subseteq \Ann_n(A), ~N \subseteq \Ann_n(B)$ of $A,~B$ respectively and brace isomorphisms $\xi:A/L\to B/N$ and $\theta: \Gamma_{n+1}(A) \to \Gamma_{n+1}(B)$ such that the following  diagrams commute
 \begin{center}
\begin{tikzcd}\nonumber
(A/L)^{n+1} \arrow [d, "(\xi)^{n+1}"]   \arrow[r, "\phi_m^A"]  & \Gamma_{n+1}(A) \arrow[d, "\theta"]\\
(A/N)^{n+1}    \arrow[r, "\phi_m^B"]  & \Gamma_{n+1}(B)
\end{tikzcd}
\end{center}
for all $m \in M$ with $\deg(m) = n$.
\end{thm}

\begin{proof}
Proof goes along the same lines as the the proof of Theorem \ref{thm:Tappe}. 
\end{proof}

The verification of the following result is left as an a exercise for the reader.

\begin{lemma}\label{lem:ann-gamma-direct}
Let $A = B \times C$ be the direct product of skew braces $B$ and $C$. Then for $n\ge 0$, the following hold true:
\begin{enumerate}
\item $\Ann_n(A)=\Ann_n(B) \times \Ann_n(C)$.
\item $\Gamma_{n+1}(A) = \Gamma_{n+1}(B) \times \Gamma_{n+1}(C)$.
\end{enumerate}
\end{lemma}
%\begin{proof}
%It is obvious for $n=0$. Now let $(x,y)\in \Ann(A)$, then for $(a,b)\in A$ we have
%\begin{align}
%[(x,y),(a,b)]^\cdot &=(x,y)\ast (a,b)=(a,b)\ast (x,y)=(1,1),\label{eq1}
%\end{align}
%thus
%\begin{align}
%([x,a]^\cdot,[y,b]^\cdot) &=(x\ast a,y\ast b)=(a\ast x,b\ast y)=(1,1).\label{eq2}
%\end{align}
%This gives $x\in\Ann(I)$ and $y\in \Ann(J)$. Also \eqref{eq2} implies \eqref{eq1}, hence $\Ann(A)=\Ann(I)\times \Ann(J)$. Note that
%$$\Ann_{n+1}(A)=\{x\in A~|~[x,a]^\cdot,~x\ast a,~a\ast x\in\Ann_{n}(A),~\forall~a\in A\}.$$ So the result follows from induction.

%For $n=1$, $\gamma_1(A)=\{1\}$, for any skew brace $A$. Now by induction, we have
%\begin{align*}
%\Gamma_{n+1}(A) &=\langle (a,b)\ast (x,y),~(x,y)\ast (a,b),~[(x,y),(a,b)]^\cdot~|~(a,b)\in A, ~(x,y)\in\Gamma_n(A)\rangle^\cdot\\
% &=\langle (a,b)\ast (x,y),~(x,y)\ast (a,b),~[(x,y),(a,b)]^\cdot~|~(a,b)\in A, ~(x,y)\in\Gamma_n(I)\times \Gamma_n(J)\rangle^\cdot\\
% &=\langle (a\ast x,b\ast y),~(x\ast a,y\ast b),~([x,a]^\cdot,[y,b]^\cdot)~|~(a,b)\in A, ~(x,y)\in\Gamma_n(I)\times \Gamma_n(J)\rangle^\cdot\\
%  &=\{ (x,y)~|~(x,y)\in\Gamma_{n+1}(I)\times \Gamma_{n+1}(J)\}\\
%  &=\Gamma_{n+1}(I)\times \Gamma_{n+1}(J).
%\end{align*}
%\end{proof}

The following observation is immediate from the definitions.

\begin{lemma}\label{lem:nilpotentSkew}
Any nilpotent skew brace $A$ with nilpotency class at most $ n$, lies in  $\mathcal{I}_{n}$.  
\end{lemma}

The next result gives some sufficient conditions so that quotients and direct products of skew braces in $\mathcal{I}_{n}$ fall back into it.
\begin{lemma}\label{lem:InI_n}
Let $A,B \in \mathcal{I}_{n}$, $I$ be an ideal of $A$ and $K$ be a sub-skew brace of $B$ for a given positive integer $n$. Then $A \times B \in \mathcal{I}_{n}$. Furthermore 
\begin{enumerate}
\item  if $\Gamma_{n+1}(A) \cap I = 1$, then $A/I \in \mathcal{I}_{n}$.
\item  if $B = K \Ann_n(B)$, then  $\Ann_n(K) = K \cap \Ann_n(B)$, which in turn implies that $K \in \mathcal{I}_{n}$.
\end{enumerate}
\end{lemma}
\begin{proof}
The  assertion on the direct product easily holds using Lemma \ref{lem:ann-gamma-direct}. Assume that $\Gamma_{n+1}(A) \cap I = 1$. We claim that $\Ann_n(A/I) = \Ann_n(A)I/I$. Let $Ia \in \Ann_n(A/I)$. Then $m(a_1a_2 \cdots a_n a) \in I$ for all $a_i \in A$ and $m \in M$ with $\deg(m) = n$. Also $m(a_1a_2 \cdots a_n a) \in \Gamma_{n+1}(A)$. Hence $m(a_1a_2 \cdots a_n a) =1$, which implies that $a \in \Ann_n(A)$, and, the other inclusion being obvious,  the claim holds true. That $A/I \in \mathcal{I}_{n}$ holds is very easy to see now.

Now assume that $B = K \Ann_n(B)$.  It is clear that $K \cap \Ann_n(A) \subseteq \Ann_n(K)$. Let $x \in \Ann_n(K)$.  Then for any $k_i \in K$, $1 \le i \le n$ and all $m \in M$ with $\deg(m) = n$, we have $m(k_1k_2 \cdots k_n x)=1$.
Let $b_i$, $1 \le i \le n$, be arbitrary elements of $B$. Then there exist $k_i \in K$ and $u_i \in \Ann_n(B)$ such that $b_i = k_i u_i$. Since $B \in \mathcal{I}_n$, we have $m(b_1 b_2 \cdots b_n x) =  m(k_1k_2 \cdots k_n x) = 1$. Hence $x \in \Ann_n(B)$. The proof is now complete.
\end{proof}

We now prove the embedding theorem (Theorem D), which is analogous to the main theorem of \cite{Waall2004}.

\begin{thm}\label{thm:Waall}
Let $A, B$ be two $n$-isoclinic ($n \ge 1$) symmetric skew braces and $(\xi,\theta)$ be a brace $n$-isoclinism. Then there exists a symmetric skew brace $R := \tilde{A} \times N$ with the following properties:
\begin{enumerate}
\item $N$ is a nilpotent skew brace of nilpotency class $n$;
\item $\tilde{A}$ is brace isomorphic to $A$, and $R=\bar{A}\Ann_n(R)$;
\item $R$ admits a sub-skew brace $\bar{B}$  isomorphic to $B$ such that  $R=\bar{B}\Ann_n(R)$,
\end{enumerate}
where $\bar{A}:=\tilde{A}\times \{1\}$. In particular, both $A$ and $B$ are $n$-isoclinic to $R$.
\end{thm}
\begin{proof}
We proceed by induction.  By Theorem \ref{thm:embed} and  Remark \ref{rmk:embed1}, the result holds for $n=1$ in the class of symmetric skew braces.
The fiber product
$$C:=\{(a, b) \mid a \in A,~ b \in B,~\xi(\Ann_n(A)a) = \Ann_n(B)b\}$$
of $A$ and  $B$ over $A/\Ann_n(A)$ is clearly a symmetric sub-skew brace of the symmetric skew brace $A \times B$. Set $N_1:=\Ann_n(A)\times \{1\}$ and $N_2:=\{1\}\times \Ann_n(B)$, which are  ideals of $C$.  Consider the natural projections $C\twoheadrightarrow A_1$ and $C\twoheadrightarrow A_2$. Since $A$ and $B$ are $n$-isoclinic, it follows that  $A$ and $B$ are brace isomorphic to $C/{N_2}$ and $C/{N_1}$ respectively. It follows from the proof of Theorem \ref{thm:isoclin-surj} that  
\begin{equation}\label{intersection}
N_1 \cap \; \Gamma_{n+1}(C) = 1 =  N_2 \cap \; \Gamma_{n+1}(C).
\end{equation}

 Now consider the brace 
$$W := C/{N_2}\times C/{\Gamma_{n+1}(C)}$$
and the subset 
$$V:=\{(N_2x,\Gamma_{n+1}(C)x)~|~x\in N_1\cap \Gamma_n(C)\}$$ of $W$. It is easy to see that $V$ is a sub-skew brace of $W$.
Since $N_1\cap \Gamma_n(C)$ is an ideal, for $c\in C$ and $x\in N_1\cap \Gamma_n(C)$, we have $cxc^{-1},~c\circ x\circ \bar{c},~\lambda_c(x)\in N_1\cap \Gamma_n(C)$. Hence  by Remark \ref{rmk1} and  the definition of  $\Gamma_{n+1}(C)$, we get 
$$\Gamma_{n+1}(C) (c\circ x\circ \bar{c}) = \Gamma_{n+1}(C)x = \Gamma_{n+1}(C)cxc^{-1} = \Gamma_{n+1}(C)\lambda_c(x).$$
This yields that $V$ is an ideal of $W$. Now consider the skew brace $Y:=W/V$. 

Let us now define
\begin{align*}
\rho_1 &:C\to Y,\hspace{1cm} c \mapsto V \big(N_2 c,\Gamma_{n+1}(C)\big),\\
\rho_2 &:C\to Y,\hspace{1cm} c \mapsto V \big(N_2 c,\Gamma_{n+1}(C) c\big).
\end{align*}
It is clear that these maps are brace homomorphisms and, using \eqref{intersection},  we have $\ker (\rho_1) = N_2$ and $\ker (\rho_2) = N_1 \cap \Gamma_{n}(C)$. Thus $\rho_1(C)$ is brace isomorphic to $A$ and $\rho_2(C)$ is brace isomorphic to $C/{N_1\cap \Gamma_{n}(C)}$.

Since  $A, B$ are symmetric skew braces, it follows that   $C/{N_1}$ and $C/{N_1\cap \Gamma_{n}(C)}$ are symmetric, thus $C/{N_1},~C/{N_1\cap \Gamma_{n}(C)}\in \mathcal{I}_{n-1}$. By Proposition \ref{prop7}(2), we have that $C/{N_1}$ is $(n-1)$-isoclinic to $C/{N_1\cap \Gamma_{n}(C)}$,  which implies that  $B$ is $(n-1)$-isoclinic to $\rho_2(C)$. Now, by the inductive step,  there exists a symmetric  skew brace $U := \tilde{D}_2 \times K_{n-1}$, where $K_{n-1}$ is a nilpotent skew brace of nilpotency class $n-1$, $\overline{D}_2 := \tilde{D}_2 \times \{1\}$ is brace isomorphic to $\rho_2(C)$, and there is a sub skew brace $\overline{E}_2$ of $U$ which is brace isomorphic to $B$  such that 
$$\overline{E}_2 \Ann_{n-1}(U) = U = \overline{D}_2 \Ann_{n-1}(U).$$ 

Set $K_n := LV/V$, where $L := \{(N_2, \Gamma_{n+1}(C)c) \mid c \in C\}$, which is symmetric by the construction. Observe that  $Y = \rho_1(C)K_n$, where $\rho_1(C)$ and $K_n$ are ideals of $Y$ such that  $\rho_1(C) \cap K_n = 1$;  hence  it follows that 
$$Y = \rho_1(C) \times K_n.$$
  Note that $K_n$ is a nilpotent skew brace of nilpotency class $n$, thus by Lemma \ref{lem:ann-gamma-direct}, we have
\begin{equation}\label{eq:em2}
\Ann_n(Y)=\Ann_n(\rho_1(C))\Ann_n(K_n)=\Ann_n(\rho_1(C))K_n.
\end{equation}
Since $Y=\rho_1(C)K_n$, we get
\begin{equation*}\label{eq:em3}
Y=\rho_1(C)\Ann_n(Y).
\end{equation*}

We now finally set
$$R:=\rho_1(C)\times K_n\times K_{n-1},$$
which is a symmetric skew brace.
Let  $\overline{\rho_1(C)}$, $\overline{K}_{n}$ and  $\overline{K}_{n-1}$, respectivey,  be the isomorphic copies  of $\rho_1(C)$, $ K_n$, and  $K_{n-1}$ such that $R$ is the internal direct product
$$R = \overline{\rho_1(C)} ~ \overline{K}_n ~ \overline{K}_{n-1}.$$
Let $\iota : Y \to R$ be the natural embedding so that $\iota(Y) =  \overline{\rho_1(C)}  ~ \overline{K}_n$.

It is clear that  $\overline{\rho_1(C)}$ is brace isomorphic to $A$, and $\overline{K}_{n}$ and $\overline{K}_{n-1}$ are  nilpotent skew braces of nilpotency class  $n$ and $n-1$ respectively. 
%In this setting, we get
%\begin{equation}\label{eq:em7}
%R=\overline{\rho_1(C)} \times \overline{K}_{n} \times \overline{K}_{n-1}.
%\end{equation}
Now by Lemma \ref{lem:ann-gamma-direct}, we get
\begin{equation}
\Ann_n(R) = \Ann_n(\overline{\rho_1(C)})~\Ann_n(\overline{K}_{n}~\overline{K}_{n-1}) = \Ann_n(\overline{\rho_1(C)})~\overline{K}_{n}~\overline{K}_{n-1},
\end{equation}
which implies that 
\begin{equation}\label{eq:em6}
R=\overline{\rho_1(C)}\Ann_n(R).
\end{equation}
From \eqref{eq:em2}, we get
\begin{equation}\label{eq:em4}
\iota(\Ann_n(Y)) = \iota(\Ann_n(\rho_1(C))K_n) = \Ann_n(\overline{\rho_1(C)}~ \overline{K}_n)  = \Ann_n(\iota(Y)).
\end{equation}

Recall that $\overline{E}_2 \Ann_{n-1}(U)=U=\overline{D}_2~\widetilde{K}_{n-1}$, where $\widetilde{K}_{n-1}:=\{1\} \times K_{n-1}$. Also we have a brace isomorphism $\tau:\overline{D_2} \to \rho_2(C)$. This yields an injective brace homomorphism
$$\bar{\tau}:\overline{D}_2\xrightarrow \tau \rho_2(C)\xrightarrow{\iota|_{\rho_2(C)}}R.$$
Now, $U$ being the internal direct product of $\overline{D}_2$ and $\widetilde{K}_{n-1}$,  every element $u \in U$ can be uniquely expressed  in the form $u=(s,1)(1,t)$, where $(s,1)\in\overline{D}_2$ and $(1,t)\in\widetilde{K}_{n-1}$. Now consider the map
$$\eta:U\to R,\hspace{1cm}u\mapsto (\bar{\tau}(s,1), 1) (1,1,t).$$
One can  now naturally see that $\eta$ is an injective brace homomorphism.

 Thus we obtain
$$ \eta(U) = \eta\big(\overline{D}_2 \widetilde{K}_{n-1}\big) =  \iota(\rho_2(C))\overline{K}_{n-1}$$
and 
$$\eta(U)= \eta\big(\overline{E}_2\Ann_{n-1}(U)\big) = \eta(\overline{E}_2)\Ann_{n-1}(\eta(U)),$$
which gives 
$$\iota(\rho_2(C))\overline{K}_{n-1} = \eta(\overline{E}_2)\Ann_{n-1}(\eta(U)).$$
Now, using various equalities obtained above, we get
\begin{align*}
R &=\overline{\rho_1(C)}~\overline{K}_n~\overline{K}_{n-1} =\iota(Y)\overline{K}_{n-1}\\
&=\iota\big(\rho_2(C)\Ann_n(Y)\big)\overline{K}_{n-1}\\
&=\iota(\rho_2(C))\Ann_n(\iota(Y))\overline{K}_{n-1}\\
&=\iota(\rho_2(C))\overline{K}_{n-1}\Ann_n(\iota(Y)) \mbox{ (as  $[\Ann_n(\iota(Y)),\overline{K}_{n-1}]^\cdot = 1$)}\\
&=\eta(\overline{E}_2) \Ann_{n-1}(\eta(U) \Ann_n(\overline{\rho_1(C)})\overline{K}_n\\
&=\eta(\overline{E}_2) \Ann_{n-1}\big(\iota(\rho_2(C))\big)\overline{K}_{n-1 }\Ann_n(\overline{\rho_1(C)})\overline{K}_n\\
&\le \eta(\overline{E}_2) \Ann_{n} \big(\iota(\rho_2(C))\big)\overline{K}_{n-1} \Ann_n(\overline{\rho_1(C)})\overline{K}_n\\
&=\eta(\overline{E}_2) \big(\Ann_{n}(\iota(Y)) \cap \iota(\rho_2(C))\big) \overline{K}_{n-1}\Ann_n(\overline{\rho_1(C)})\overline{K}_n \mbox{ ~(by Lemma \ref{lem:InI_n}(2))}\\
&=\eta(\overline{E}_2)\big(\Ann_{n}\big(\overline{\rho_1(C)}~\overline{K}_n\big)\cap\iota(\rho_2(C))\big)\overline{K}_{n-1} \Ann_n(\overline{\rho_1(C)})\overline{K}_n\\
&=\eta(\overline{E}_2)\big(\Ann_{n}(\overline{\rho_1(C)})\overline{K}_n \cap \iota(\rho_2(C))\big) \big(\Ann_n(\overline{\rho_1(C)})\overline{K}_n\big) \overline{K}_{n-1}\\
&=\eta(\overline{E}_2) \Ann_n(\overline{\rho_1(C)})\overline{K}_n \overline{K}_{n-1}\\
&=\eta(\overline{E}_2)\Ann_n(R)\\
&\le R.
\end{align*}

This, using \eqref{eq:em6},  finally yields
$$\overline{\rho_1(C)}\Ann_n(R)=\eta(\overline{E}_2)\Ann_n(R) = R = \rho_1(C) \times (K_n \times K_{n-1}),$$ 
where  $\overline{\rho_1(C)} \cong \rho_1(C)\cong A$ and  $\eta(\overline{E}_2) \cong \overline{E}_2 \cong B$  are brace isomorphisms and 
$K_n \times K_{n-1}$ is a nilpotent skew brace having  nilpotency  class $n$. The skew brace $R$ is a direct product of $N := K_n\times K_{n-1}$ and $A$, thus by Lemma \ref{lem:nilpotentSkew} and \ref{lem:InI_n}, $R\in\mathcal{I}_n$. That $A$ and $B$ are brace $n$-isoclinic to $R$ follows from Proposition \ref{prop7}(1). The proof is now complete.
\end{proof}

As a consequence, we get the following interesting result.

\begin{cor}
Let $A$ and $B$ be $n$-isoclinic symmetric skew braces. Then  $(A, \cdot)$ is $n$-isoclinic to  $(B, \cdot)$  and  $(A, \circ)$ is $n$-isoclinic to $(B, \circ)$.
\end{cor}
\begin{proof}
Let $G$ be a symmetric skew brace and $a, x \in G$. Then from the proof of  Lemma \ref{lemma2}, we get
$$[x,a]^\circ=\lambda_x\Big(\bar{x}^{-1}\lambda_a([\bar{a}^{-1},\bar{x}]^\cdot)\bar{x}(\bar{x}^{-1}(a\ast \bar{x})\bar{x})\lambda_a(\bar{a}^{-1}(\bar{x}\ast a)\bar{a})\Big).$$
Thus from the definition of the ideal $\Ann_r(G)$, an easy induction gives $\Ann_r(G)\subseteq Z_r(G,\cdot)\cap Z_r(G,\circ)$ for any skew brace $G$, where $r\ge 0$. With this observation and Lemma \ref{lemma1},  Theorem  \ref{thm:Waall} confirms that we can assume $A,~B$ to be sub-skew braces of a symmetric skew brace $R$ with
$$(R,\cdot)=(A,\cdot) Z_n(R,\cdot)=(B,\cdot) Z_n(R,\cdot)$$
and 
$$(R,\circ) = (A,\circ)\circ Z_n(R,\circ)=(B,\circ)\circ Z_n(R,\circ).$$
Now the result follows from Lemma 3.5  of \cite{Hekster}.
\end{proof}

We conclude this section with the remark that  we gave a serious try to the preceding result for all skew braces in $\tilde{\mathcal{I}}_n$, but could not succeed. One perhaps needs to come up with  different ideas to get the success.

%%%%%%%%%%%%%%%%%%%%%%%%%%%%%%%%%%%%%%%%%%%%%%%%%%%%%%%%%%%%

\section{Isoclinism invariance}

Let $(A, \cdot, \circ)$ be a skew brace. Then for any element $a \in A$, the map $\lambda^{op}_a$, defined by $\lambda^{op}_a(b) := \iota_a(\lambda_a(b))$ for all $b \in A$, is an automorphism $(A, \cdot)$, where $\iota_a$ denotes the inner automorphism of $(A, \cdot)$ induced by $a$. Also the map $\lambda^{op} : (A, \circ) \to  \Aut(A, \cdot)$ such that $a \mapsto \lambda^{op}_a$ is a group homomorphism, that is, $\lambda^{op} $ is a group action of $(A, \circ)$ on $(A, \cdot)$.  Then we naturally get a group $\Lambda_A := (A, \cdot) \rtimes_{\lambda^{op}} (A, \circ)$, the semi-direct product of $(A, \circ)$ by $(A, \cdot)$ under the action $\lambda^{op}$ with the composition law given by $(a_1, b_1) (a_2, b_2) = \big(a_1\cdot \lambda^{op}_{b_1}(a_2),~b_1 \circ b_2\big)$ for $(a_1, b_1), (a_2, b_2) \in \Lambda_A$.

 Let $A$ and $B$ be two skew braces. Then, for a brace homomorphism $f : A \to B$, the map $\bar{f} : \Lambda_A \to \Lambda_B$, given by $\bar{f}(a_1, b_1) = (f(a_1), f(b_1))$ for all $(a_1, b_1)\in \Lambda_A$, turns out to be a group homomorphism. Note that $\mathcal{SB}$, the class of all skew braces, forms a category.  Let $\mathcal{G}$ be  the category of groups. Then we have

\begin{prop}\label{prop:func}
The map $\Lambda: \mathcal{SB} \to \mathcal{G}$, given by  $A \mapsto \Lambda_A$, is a functor.  Moreover, brace isomorphisms are mapped to group isomorphisms.
\end{prop}
%\begin{proof}
%As explained above, for any skew brace $(A, \cdot, \circ)$, we naturally get  the group $\Lambda_A$.
%Let $f:(A, \cdot, \circ) \to (B, \cdot, \circ)$ be a  brace homomorphism. Then it follows that
%\begin{equation}
%f(\lambda^{op}_{a_1}(b_1))=\lambda^{op}_{f(a_1)}(f(b_1))\label{eq:fLamb1}
%\end{equation}
%for all $a_1, b_1\in A$.
% Now consider the map $\bar{f} : \Lambda_A \to \Lambda_B$ given by 
% $$\bar{f}(a_1, b_1) = (f(a_1), f(b_1))$$
%for all  $(a_1, b_1)\in \Lambda_A$. By \eqref{eq:fLamb1}, we get
%\begin{align*}
%\bar{f}\big((a_1, b_1)(a_2,b_2)\big) &=\bar{f}\big(a_1\cdot \lambda^{op}_{b_1}(a_2),~b_1 \circ b_2\big)\\
%&=\big(f(a_1\cdot \lambda^{op}_{b_1}(a_2)),~f(b_1 \circ b_2)\big)\\
%&=\big(f(a_1) \cdot \lambda^{op}_{f(b_1)}(f(a_2)),~f(b_1)\circ f(b_2)\big)\\
%&=\bar{f}(a_1, b_1) \bar{f}(a_2, b_2).
%\end{align*}
%This proves that $\bar{f}$ is a group homomorphism, which completes the proof.
%\end{proof}

The following definition of representation of a skew brace is introduced in \cite{LV24}.
A  triple $(V,\beta,\rho)$ is said to be a {\em brace representation} of a skew brace $(A, \cdot, \circ)$ if\\
(i) $V$ is a vector space over a field $k$;\\
(ii) $\beta: (A, \cdot) \to GL(V)$ is a representation of $(A, \cdot)$;\\
(iii) $\rho: (A, \circ) \to GL(V)$ is a representation of $(A, \circ)$\\
such that 
\begin{equation}
\beta(\lambda^{op}_a(b))=\rho(a)\beta(b)\rho(a)^{-1}
\end{equation}
holds for all  $a, b \in A$.

For a given brace representation  $(V,\beta,\rho)$ of $(A, \cdot, \circ)$, we can naturally construct a representation $(V,\phi_{(\beta,\rho)})$ of the group $\Lambda_A$ by defining 
$$\phi_{(\beta,\rho)}:\Lambda_A\to GL(V),\hspace{.5cm}\phi_{(\beta,\rho)}(a,b)=\beta(a)\rho(b).$$ 
Conversely, if $\phi:\Lambda_A\to GL(V)$ is a representation of the group $\Lambda_A$, then we can define $\beta,\rho$ by setting 
\begin{align*}
\beta &: (A,\cdot)\to GL(V),\hspace{1cm}\beta(a)=\phi(a, 1)\\
\rho &: (A,\circ)\to GL(V),\hspace{1cm}\rho(b)=\phi(1, b)
\end{align*}
so that  $(V,\beta,\rho)$ is a brace representation of $(A, \cdot, \circ)$.

Two representations $(V_1,\beta_1,\rho_1)$ and $(V_2,\beta_2,\rho_2)$ of $(A, \cdot, \circ)$ are said to  be {\em equivalent} if the corresponding representations $(V_1,\phi_{(\beta_1,\rho_1)})$  and $(V_2, \phi_{(\beta_2,\rho_2)})$ of $\Lambda_A$ are equivalent.
We have proved 

\begin{prop}
The study of the brace representations of skew braces $(A, \cdot, \circ)$ is equivalent to the study of the representations of corresponding natural semi-direct products $\Lambda_A$.
\end{prop}

\subsection{$n$-isoclinism of natural semi-direct product.}
Let $I$ be an ideal of a skew brace $A := (A, \cdot, \circ)$. Then  $\Lambda_I := I \rtimes_{\lambda^{op}} I$ is a normal subgroup of $\Lambda_A$.  An  easy check reveals that $\Lambda_{\Ann(A)}\subseteq \Z(\Lambda_A)$. 
The following result follows from  \cite[Proposition 3.3]{GG09} (also reproved in \cite[Theorem 3.7]{RSU}). 

\begin{lemma}\label{lemma:comm}
For any skew brace $A$, $\gamma_2(\Lambda_A) = \Gamma_2(A) \rtimes_{\lambda^{op}} \gamma_2(A, \circ)$.
\end{lemma}

We now generalize this for the further terms of the lower central series.
 Set $L_1(A) = A$ and define 
\begin{eqnarray*}
K_n(A) &:=& \gen{\lambda_u(b)b^{-1} \mid u \in \gamma_{n-1}(A, \circ), b \in A}^{\cdot},\\
A_n &:=& \gen{\lambda_a(v)v^{-1} \mid a \in A, v \in L_{n-1}(A)}^{\cdot},\\
L_n(A) &:=& \gen{K_n(A), A_n, [A,L_{n-1}(A)]^{\cdot}}^{\cdot}.
\end{eqnarray*}

It follows from \cite[Theorem 1.1]{GP20} that $L_n(A)$ is a left ideal of $A$ and it is a normal subgroup of $(A, \cdot)$.  Notice that $L_2(A) = \gen{A^2, [A, A]^{\cdot}}^{\cdot} = \Gamma_2(A)$ and $L_n(A) \subseteq \Gamma_n(A)$ for all $n \ge 1$.  We now prove

\begin{lemma}
$\gamma_n(A, \circ)$ is contained in  $L_n(A)$ for all $n \ge 1$. 
\end{lemma}
\begin{proof}
Since $\Gamma_2(A) = L_2(A)$, it follows from Lemma \ref{lemma2} that $\gamma_2(A, \circ) \subseteq  L_2(A)$. By induction assume that  $\gamma_{n-1}(A, \circ) \subseteq L_{n-1}(A)$. Let $u = [v, a]^{\circ} \in \gamma_{n}(A, \circ)$, where $v \in \gamma_{n-1}(A, \circ)$ and $a \in A$. Now using the computation as done for  the $n=2$ case, we have
\begin{equation*}
u = [v, a]^{\circ} = \lambda_v\big(\bar{v}^{-1}\bar{a}^{-1}  \bar{v}  \bar{a} (c^{-1} \lambda_a(c)) \lambda_a (\bar{a}^{-1} \lambda_{\bar{v}}( \bar{a}))\big),
\end{equation*}
where $c:=\bar{a}^{-1}  \bar{v}  \bar{a}$.  Notice that  $[ \bar{v}^{-1}, \bar{a}^{-1}]^{\cdot}  = ( [\bar{a}^{-1},  \bar{v}^{-1}]^{\cdot})^{-1} \in [A,   L_{n-1}(A)]^{\cdot}$, $c^{-1} \lambda_a(c) =  (\lambda_a(c^{-1}) c)^{-1} \in A_n$ as $c \in L_{n-1}(A)$, and $\bar{a}^{-1} \lambda_{\bar{v}}( \bar{a})  = ( \lambda_{\bar{v}^{-1}} \bar{a})^{-1} \in K_n(A)$. Hence $\gamma_{n}(A, \circ) \subseteq L_{n}(A)$, and the proof is complete. 
\end{proof}

The following can be verified easily.

\begin{lemma}
For any skew  brace $A$,  the sets $K_n(A)$ and $\gamma_n(A, \cdot)$ are left ideals of $A$, and both are normal in $(A, \cdot)$.
\end{lemma}

The next two results follow   from \cite[Theorem 1.1]{GP20}.
\begin{lemma}\label{lemma3}
For any skew  brace $A$, we have $[L_n(A), \Ann_n(A)]^{\cdot} = 1$.
\end{lemma}

\begin{lemma}\label{lemma:comm:n}
For any skew brace $A$ and $n \ge 2$, $\gamma_n(\Lambda_A) =  L_n(A) \rtimes_{\lambda^{op}} \gamma_n(A, \circ)$.
\end{lemma}

Notice that the set of commutator identities satisfied for an element of $A$ to fall in $\Ann_n(A)$ contains the set of commutator identities satisfied for an element of $\Lambda_A$ to fall in $\Z_n(\Lambda_A)$. Thus it is not difficult to prove that 
$\Lambda_{\Ann_n(A)} \subseteq \Z_n(\Lambda_A)$. We now prove that, under certain conditions,  $n$-isoclinism is invariant under the functor $A \mapsto \Lambda_A$. This is Theorem E of the introduction.

\begin{thm}
Let $A$ and $B$ be two  $n$-isoclinic  skew braces in $\mathcal{I}_n$ such that $A_{(n)} = \Gamma_{n+1}(A)$ and $B_{(n)} = \Gamma_{n+1}(B)$. Then the groups $\Lambda_A$ and $\Lambda_B$ are $n$-isoclinic. 
\end{thm}
\begin{proof}
We begin observing that the map
\begin{equation}
\psi^n_A: \Lambda_A \to \Lambda_{A/{\Ann_n(A)}},\hspace{.5cm} (a,b) \mapsto \big( \Ann_n(A) a,  \Ann_n(A) \circ b\big),
\end{equation}
is a surjective group homomorphism with $\Ker(\psi^n_A) =\Lambda_{\Ann_n(A)}$. Thus $\psi^n_A$ induces an isomorphism $\bar{\psi}^n_A:\Lambda_A/\Lambda_{\Ann_n(A)} \to \Lambda_{A/{\Ann_n(A)}}$.

Let $(\xi,\theta)$ be a brace $n$-isoclinism  between $A$ and $B$. As $\xi$ and $\theta$ are brace isomorphisms, by Proposition \ref{prop:func}, we get induced isomorphisms $\Lambda(\xi): \Lambda_{A/{\Ann_n(A)}}\to \Lambda_{B/{\Ann_n(B)}}$ and $\Lambda(\theta):\Lambda_{\Gamma_{n+1}(A)}\to \Lambda_{\Gamma_{n+1}(B)}$, respectively, given by 
\begin{align*}
\Lambda(\xi)\big( \Ann_n(A) a_1,  \Ann_n(A) a_2\big) &=\big(\xi(\Ann_n(A) a_1),~\xi(\Ann_n(A) a_2)\big)\\
\Lambda(\theta)(u_1, u_2) &=(\theta(u_1),\theta(u_2))
\end{align*}
for all $a_1, a_2 \in A$ and $u_1, u_2 \in \Gamma_{n+1}(A)$.
We now set
$$\alpha:= (\bar{\psi}_B^{n})^{-1} \Lambda(\xi)\bar{\psi}_A^n: \Lambda_A/\Lambda_{\Ann_n(A)}\to \Lambda_{B}/{\Lambda_{\Ann_n(B)}},$$
which is a brace isomorphism.
By  Lemma \ref{lemma:comm:n} we have $\gamma_{n+1}(\Lambda_{A}) =   L_{n+1}(A) \rtimes_{\lambda^{op}} \gamma_{n+1}(A, \circ) \subseteq \Lambda_{\Gamma_{n+1}(A)}$ and $\gamma_{n+1}(\Lambda_B) = L_{n+1}(B) \rtimes_{\lambda^{op}} \gamma_{n+1}(B, \circ)\subseteq  \Lambda_{\Gamma_{n+1}(B)}$.  Since $(\xi, \theta)$ is a brace isoclinism, it turns out that $\beta := \Lambda(\theta)|_{\gamma_{n+1}(\Lambda_A)}$ is a group isomorphism from $\gamma_{n+1}(\Lambda_{A})$ to $\gamma_{n+1}(\Lambda_B)$. A straightforward computation now shows that the following diagram commutes:

\begin{equation*}
\begin{tikzcd}
(\Lambda_A/{\Lambda_{\Ann_n(A)}})^{n+1} \arrow[d,"\alpha \times \alpha"]  \arrow[r, " "]  & \gamma_{n+1}(\Lambda_{A}) \arrow[d,"\beta"] \\
(\Lambda_B/{\Lambda_{\Ann_n(B)}})^{n+1}   \arrow[r, " "]  &\gamma_{n+1}(\Lambda_B).
\end{tikzcd}
\end{equation*}
As noticed above, $\Lambda_{\Ann_n(A)}$ and $\Lambda_{\Ann_n(B)}$ are subgroups of $\Z_n(\Lambda_A)$ and $\Z_n(\Lambda_B)$ respectively. It now follows by \cite[Lemma 3.11]{Hekster} that $(\alpha,\beta)$ is an  isoclinism, which completes the proof.
\end{proof}

Note that the whole class of skew braces $A$ falls in $\mathcal{I}_1$ and $A_{(1)} = \Gamma_2(A)$. So we readily get 
\begin{cor}
Let $A$ and $B$ be two  isoclinic  skew braces. Then the groups $\Lambda_A$ and $\Lambda_B$ are isoclinic. 
\end{cor}

The concept of soluble and supersoluble skew left braces was introduced and investigated in \cite{BEFPT24, BEJP24}. It is proved that for a supersoluble skew brace $A$, if $A/\Ann_n(A)$ is finite, then $\Gamma_{n+1}(A)$ is finite (\cite[Theorem 3.47]{BEFPT24}). The group theory result holds for arbitrary $n$ and the full class of groups, and it  is due to R. Baer \cite{Baer52}. As a consequence of the  discussion in this section, we get the following result.

\begin{prop}\label{gamma2:finite}
Let, for a given integer $n \ge 1$, $A$ be a skew brace such that $A /\Ann_n(A)$ is finite. Then $L_{n+1}(A)$  is finite.  In particular, for $n = 1$, we get:  if $A /\Ann(A)$ is finite, then $\Gamma_2(A)$ is finite.
\end{prop}
 
\begin{proof}
If $A /\Ann_n(A)$ is finite, then $\Lambda_A/\Lambda_{\Ann_n(A)}$ is finite. Since $\Lambda_{\Ann_n(A)} \subseteq \Z_n(\Lambda_A)$, it follows that $\Lambda_A/\Z_n(\Lambda_A)$ is finite. Hence $\gamma_{n+1}(\Lambda_A)$ is finite. The proof now follows from Lemma \ref{lemma:comm:n}. The second assertion holds by the fact that $\Gamma_2(A) = L_{2}(A)$.
\end{proof}

The case $n=1$ of the preceding result also follows from  \cite[Theorem 5.4]{JKAV21}.

\subsection{Other invariance results.}
We finally prove that various nilpotencies of  skew braces are invariant under $n$-isoclinism.    We actually prove little general results, which will imply invariance under $n$-isoclinism.  

 A skew brace $A$ is \textit{left nilpotent}, respectively \textit{right nilpotent}, if there exists $n \in \N$ such that $A^n = 1$, respectively $A^{(n)} = 1$.   For $A$ a skew brace, define $A^{[1]} = A_1 = A$ and for $n \geq 2$, iteratively define
    \begin{eqnarray*}
        A^{[n]} &=& \Big\langle \bigcup_{i=1}^{n-1} A^{[i]} * A^{[n-i]}
\Big\rangle^\cdot, \\
        A_n &=& A_{n-1} * A_{n-1}.
    \end{eqnarray*}
 Note that   these are descending series' of left ideals of $A$.  If there exists $n \geq 1$ such that $A^{[n]} = 1$, then $A$ is said to be 
\textit{strongly nilpotent}.    If there exists $n \geq 1$ such that $A_n = \{1\}$, then $A$ is said to be  \textit{$*$-soluble}.

Strong nilpotency can be characterized through left and right nilpotency, as was first proved in \cite{smoktunowicz18} for braces
and in \cite{CSV19} for skew braces in general.

\begin{thm} \label{thm:strong_left_right}
   For a skew brace  $A$, the following are equivalent:
    \begin{enumerate}
        \item $A$ is strongly nilpotent.
        \item $A$ is left and right nilpotent.
    \end{enumerate}
\end{thm}

The following result will show the invariance of the desired classes of skew braces under $n$-isoclinism in $\mathcal{I}_n$, $ n \ge 1$.
\begin{prop} \label{prop:left_right_nilpot}
    Let $A$ and $B$ be two  skew braces such that $A/\Ann_n(A)$ is brace isomorphic to   $B/\Ann_n(B)$ for some $n \ge 1$. Then the following hold:
    \begin{enumerate}
        \item $A$ is right nilpotent if and only if $B$ is right nilpotent.
        \item $A$ is left nilpotent if and only if $B$ is left nilpotent.
        \item $A$ is nilpotent if and only if $B$ is nilpotent.
        \item $A$ is strongly nilpotent if and only if $B$ is strongly nilpotent.
        \item $A$ is  $*$-soluble if and only if $B$ is $*$-soluble.
    \end{enumerate}
\end{prop}

\begin{proof}
    We only prove assertion (1); the remaining ones can be  proven analogously. Indeed, (4) follows from (1), (2) and Theorem \ref{thm:strong_left_right}.

    Suppose that $A$ is right nilpotent. Then, it follows that $A/\Ann_n(A)$ is right nilpotent (for an arbitrary surjective brace homomorphism, it follows from \cite[Lemma 2.5]{CSV19} that the image of a right nilpotent skew brace is right nilpotent). Therefore, by the given hypothesis,  $B/\Ann_n(B)$ is right nilpotent as well. Thus there exists $k \in \N$ such that $(B/\Ann_n(B))^{(k)} = \{1\}$, or equivalently,
    \[
    B^{(k)} \subset \Ann_n(B).
    \]
    By the  definition of $\Ann_n(B)$, we have
    \[
    B^{(k+1)} = B^{(k)} * B \subset \Ann_n(B) * B \subset \Ann_{n-1}(B).
    \]
    Repeating this reasoning $n$ times, it follows that $B^{(k+n)}
\subset \Ann_0(B) = 1$. Hence $B$ is right nilpotent.
\end{proof}

We conclude with two remarks and one question.

\begin{remark}
Most of the results analogous to the results of the paper \cite{Hekster} can be obtained for symmetric skew braces. The other subclass of skew braces for which the theory will go well is the class of $\lambda$-homomorphism skew braces.  It will be interesting to obtain such results for a larger subclass of $\mathcal{I}_n$.
\end{remark}

\begin{remark}
Concepts of monomial and strongly monomial skew braces can be introduced using Proposition  \ref{prop:func} by defining a  character through the representation of a skew brace defined  above in this section. Then the invariance of these concepts under $n$-isoclinism can be studied, for $n \ge 1$.
\end{remark}

We know that $\mathcal{I}_1  = \bar{\mathcal{I}}_1= \tilde{\mathcal{I}}_1 = \mathcal{SB}$ and symmetric skew braces are contained in $\tilde{\mathcal{I}}_n$ for all $n \ge 1$. The following question seems very interesting.

\vspace{.05in}

\noindent  {\bf Question.} 
How big are the classes $\mathcal{I}_n$, $\bar{\mathcal{I}}_n$  and  $\tilde{\mathcal{I}}_n$ for a given $n \ge 2$?

\vspace{.1in}
%%%%%%%%%%%%%%%%%%%%%%%%%%%%%%%%%%%%%%%%%%%%%%%%%%%%%%%%%%%%%%%%%

\noindent {\bf Acknowledgement:} The third-named author acknowledges the support of DST-SERB Grant MTR/2021/000285. He would also like to thank Leandro Vendramin and Mathematics Department of VUB, Brussels for hosting him for two weeks in March 2024, where the present work was initiated.  The second named author acknowledges partial support by the project OZR3762 of Vrije Universiteit Brussel and FWO Senior Research Project G004124N. The authors thank Leandro Vendramin for his kind help with GAP computations and constructive remarks on the presentation of the article.


\begin{thebibliography}{100} 

\bibitem{BEFPT24}
A. Ballester-Bolinches, R.  Esteban-Romero, M. Ferrara, V. Perez-Calabuig and M. Trombetti,  \emph{ Finite skew braces of square-free order and supersolubility}, Forum Math. Sigma  {\bf 12} (2024), Paper No. e39, 33 pp.

\bibitem{BEJP24}
A. Ballester-Bolinches, R. Esteban-Romero, P. Jiménez-Seral and V. Pérez-Calabuig,  \emph{Soluble skew left braces and soluble solutions of the Yang–Baxter solution},  Adv. Math. {\bf 455} (2024), Paper No. 109880, 27 pp.

\bibitem{Baer52} 
R. Baer,  \emph{Endlichkeitskriterien für Kommutatorgruppen},  Math. Ann. {\bf 124} (1952), 161-177.

\bibitem{BNY22}
V. G.  Bardakov, M.  V. Neshchadim and Manoj K. Yadav,  \emph{ $\lambda$-homomorphic skew braces}, J. Pure 	Appl. Algebra {\bf 226} (2022), no .6, Paper No. 106961, 37 pp.


\bibitem{Bioch78}
 J. C. Bioch, \emph{On  $n$-isoclinic groups}, Indag. Math. {\bf 38} (1976), no. 5, 400-407.

\bibitem{BJ23}
M. Bonatto and P. Jedlicka, \emph{Central nilpotency of skew braces}, J Algebra Appl. {\bf 22} (2023), 2350255 (16 pages).

\bibitem{CSV19} 
F. Cedo, A.  Smoktunowicz and L.  Vendramin, {\it Skew left braces of nilpotent type}, Proc. London Math. Soc., {\bf 118} (2019), 1367-1392.

\bibitem{Drinfeld92}
V. G. Drinfel'd, \emph{On some unsolved problems in quantum group theory}, Lecture Notes in Math., {\bf 1510}
Springer-Verlag, Berlin, 1992, 1-8.

\bibitem{ESS99}	
 P. Etingof,  T. Schedler and A. Soloviev. {\it Set-theoretical solutions to the quantum Yang-Baxter equation},  Duke Math. J. {\bf 100}, no. 2 (1999), 16-209.
 
 \bibitem{GAP}
The GAP Group, \emph{Groups Algorithms and Programming,}  version 4.12.2 (2018), available at  http://www.gap-system.org.

\bibitem{GV17}
 L. Guarnieri  and  L. Vendramin, \emph{ Skew braces and the Yang-Baxter equation},  Math. Comp. {\bf 86} (2017), 2519-2534.

\bibitem{GG09}
D.~L. Goncalves  and J. Guaschi,  {\it The lower central and derived series of the braid groups of the sphere},  Trans. Amer. Math. Soc. {\bf 361} (2009), no. 7, 3375-3399.

\bibitem{GP20}
J. Guaschi and  Carolina de Miranda e Pereiro,  \emph{Lower central and derived series of semi-direct products, and applications to surface braid groups}, J. Pure Appl. Algebra {\bf 224} (2020), 106309, 39 pp.
	
\bibitem{Hall40}	
P. Hall, \emph{The classification of prime power groups}, J. Reine Angew. Math. {\bf 182}  (1940) 130-141.	


\bibitem{Hekster}
N.~S. Hekster, {\em On the structure of $n$-isoclinism classes of groups}, J. Pure Appl. Algebra {\bf 40} (1986), no. 1, 63-85.
	
 \bibitem{JKAV21}
E. Jespers,  L. Kubat, A.  Van Antwerpen and L.  Vendramin, 
\emph{Radical and weight of skew braces and their applications to structure groups of solutions of the Yang-Baxter equation},
Adv. Math. {\bf 385} (2021), Paper No. 107767, 20 pp.

\bibitem{JAV23}
E. Jespers,  A.  Van Antwerpen and L.  Vendramin, 
\emph{Nilpotency of skew braces and multipermutation solutions of the Yang-Baxter equation},
Commun. Contemp. Math. {\bf 25} (2023), no. 9, Paper No. 2250064, 20 pp.

\bibitem{KT25}
Y. Kozakai and C.  Tsang, \emph{Representation theory of skew braces},  International Journal of Group Theory, {\bf 14} (2025) No. 3 (2025),  149-164.

\bibitem{LV23}
T. Letourmy and L. Vendramin, \emph{Isoclinism of skew braces}, Bull. Lond. Math. Soc. {\bf 55} (2023), no. 6, 2891-2906.

\bibitem{LV24}
T. Letourmy and L. Vendramin, \emph{Schur covers of skew braces}, Journal of Algebra {\bf 644} (2024) 609-654.
 
\bibitem{NY24}
N. Rathee and Manoj K. Yadav, {\it Cohomology, extensions and automorphisms of skew braces}, J Pure Applied Algebra {\bf 228} (2024),  no. 2, Paper No. 107462, 30 pp.

\bibitem{RSU}
N. Rathee, M. Singh and A, Udeep, \emph{Representations of skew braces}, arXiv:2408.03766.

\bibitem{Rump07}
W. Rump, \emph{Braces, radical rings, and the quantum Yang-Baxter equation},  J. Algebra {\bf 307} (2007), no. 1, 153-170.

\bibitem{smoktunowicz18}
A. Smoktunowicz. \emph{On Engel groups, nilpotent groups, rings, braces and the Yang-Baxter equation},  Trans. Amer. Math. Soc. {\bf 370} (2018), no. 9, 6535-6564.

\bibitem{Tappe76}
J. Tappe, \emph{On isoclinic groups},  Math. Z. 1{\bf 48} (1976), no. 2, 147-153.

\bibitem{VK22}
L. Vendramin and  O. Konovalov, {\it Yang-Baxter}, version 0.10.2  (2022), available at https://www.gap-system.org/Packages/yangbaxter.html.

\bibitem{Waall2004}
R. W. van der Waall, 
\emph{ On n-isoclinic embedding of n-isoclinic groups},
Indag. Mathem., N.S., {\bf 15} (2004), 595-600.

\end{thebibliography}
\end{document}